\documentclass[a4paper,11pt,oneside]{article} % ou amsart , ...
\usepackage{amsfonts,amssymb}
\usepackage{color} \usepackage{mathrsfs} \let\mathcal\mathscr
\usepackage[all,ps,cmtip]{xy}
 \usepackage[top=1in, bottom=1in, left=1in, right=1in]{geometry}

\title{\sc On the twisted octonionic eigenvalue problem and some sextics hypersurfaces related to the Cartan cubic}
\author{\sc Roland Abuaf \footnote{Rectorat de Paris, 47 rue des \'Ecoles, 75005 Paris. \textit{rabuaf@gmail.com}}}
\date{}

\usepackage[dvips]{graphicx}

\usepackage{amsfonts,amssymb}
\usepackage{color} \usepackage{mathrsfs} \let\mathcal\mathscr
\usepackage{mathenv}
\usepackage[all,ps,cmtip]{xy}

\DeclareGraphicsExtensions{eps}
\DeclareGraphicsExtensions{pdf}

\usepackage[T1]{fontenc}
\usepackage{amssymb}

\usepackage{amsmath}
\usepackage{latexsym}

\usepackage[latin1]{inputenc}

\everymath{\displaystyle}
\newtheorem{theo}{Theorem}[subsection]
\newtheorem{theo*}{Theorem}
\newtheorem{prob*}{Problem}
\newtheorem{prop*}{Proposition}
\newtheorem{quest*}{Question}

\newtheorem{rem}[theo]{Remark}

\newtheorem{prop}[theo]{Proposition}

\newtheorem{lem}[theo]{Lemma}

\def\O{\mathbb{O}}

\def\OP{\mathbb{OP}}

\def\R0{\mathrm{R^{0}}}

\def\K{\mathrm{K}}

\def\OO{\mathcal{O}}

\def\ll{\lambda}
\def\l3{\lambda_3}

\def\J{\mathrm{J}}

\def\C{\mathbb{C}}
\def\CC{\mathcal{C}}
\def\R{\mathbb{R}}
\def\SS{\mathrm{S_{ODM}}}

\def\H{\mathbb{H}}

\def\A{\mathbb{A}}

\def\k{\kappa}
\def\det{\mathrm{det}}

\newcommand*{\barrow}{\xrightarrow{\hspace*{0.8cm}}}

\newenvironment{proof}
{%\hspace*{-.6cm}
\noindent
\textit{\underline{Proof}} :\\
$\blacktriangleright\;$%
}
{\hspace{\stretch{1}}%
$\blacktriangleleft$}

\begin{document}

\maketitle

\begin{abstract}
We revisit the octonionic eigenvalue problem from a geometric perspective. In particular, we study a tautological sheaf defined on a sextic related to this problem, the Ogievetski\^i-Dray-Manogue sextic. We then define and study a twisted version of the octonionic eigenvalue problem. A new sextic arises in this setting and we study the corresponding tautological sheaf supported on it. This twisted version of the octonionic eigenvalue problem is eminently more symmetric than the original one, as reflected by the last result we prove in this paper : the automorphism group of the twisted octonionic eigenvalue problem, though not isomorphic to $\mathrm{E}_6$, acts prehomogeneously on the exceptional Jordan algebra $\J_{3}(\O)$. This is in sharp contrast with the fact that the generic orbit for the action of the automorphism group of the classical octonionic eigenvalue problem has (at least) codimension $6$ in $\J_{3}(\O)$.
\end{abstract}

\vspace{\stretch{1}}

\newpage
\tableofcontents
\begin{section}{Introduction}

One starting point of our paper is the claim made by Bhargava and Ho that there exists a << \textit{less familiar to most people}>> (as they state it) rank $8$ tautological vector bundle on the (complexified) octonionic projective plane $\mathbb{OP}^2 \subset \mathbb{P}^{26}$ (see Theorem $5.27$ of \cite{BargaHo}). Following their definition 5.21, one should believe that this vector bundle is endowed with a \textit{global faithful} right action of the (complexified) octonions $\mathbb{O}$. Let us recall the well-known construction of such a tautological bundle in the case of an associative composition algebra instead of $\O$. 

\bigskip

We denote by $\A$ one of the (complexified) algebras $\R \otimes_{\R} \C$, $\C \otimes_{\R} \C$, $\H \otimes \C$ and we let $\J_{3}(\A)$ be the space of hermitian matrices with coefficients in $\A$:
\[ \J_{3}(\A) = \left\{ \begin{pmatrix} \lambda_1 & c & \overline{b} \\ \overline{c} & \lambda_2 & a \\ b & \overline{a} & \l3 \end{pmatrix}, \ a,b,c \in \A, \ \lambda_i \in \mathbb{C}  \right\}.\] For any $A \in \J_3(\A)$, we denote by $\det_{\A}(A)$ the number: 
\begin{equation*}
\begin{split}
\det_{\A}(A) &= \lambda_1 \lambda_2 \l3 + c(ab) + \big((\overline{b})( \overline{a}) \big)(\overline{c}) - \lambda_2 b \overline{b} - \lambda_1 a \overline{a} - \l3 c \overline{c}\\
                   & =  \lambda_1 \lambda_2 \l3 + 2 \mathrm{Re}(cab)  - \lambda_2 |b|^2 - \lambda_1 |a|^2 - \l3 |c|^2,
\end{split}
\end{equation*}
where $\mathrm{Re}(a)$ is the real part of $a \in \A$ and $|a|^2$ is the square of its norm. This expression corresponds to the Sarrus expansion for a formal determinant of the $3 \times 3$ matrix $A$. We notice the remarkable that $\det_{\A}(A) \in \C$, for all $A  \in \J_3(\A)$. Finally, we denote by $\CC_{\A}$ the cubic hypersurface in $\mathbb{P}(\J_3(\A))$ given by the equation $\det_{\A} = 0$ and $X_{\A}$ the singular locus of this hypersurface. The case by case analysis is well-known (\cite{roberts} for instance):
\begin{itemize}
\item if $\A = \R \otimes \C$, then $X_{\A} = v_2(\mathbb{P}^2) \subset \mathbb{P}^5$ is the Veronese embedding of $\mathbb{P}^2$ and $\CC_{\A}$ is the determinantal cubic for $3*3$ symmetric matrices.
\item if $\A = \C \otimes \C$, then $X_{\A} =  \mathbb{P}^2 \times \mathbb{P}^2 \subset \mathbb{P}^8$ is the Segre embedding of $\mathbb{P}^2 \times \mathbb{P}^2$ and $\CC_{\A}$ is the determinantal cubic for $3*3$ general matrices.
\item if $\A = \H \otimes \C$, then $X_{\A} = \mathrm{Gr}(2,6) \subset \mathbb{P}^{14}$ is the Pl\"ucker embedding of $\mathrm{Gr}(2,6)$ and $\CC_{\A}$ is the Pfaffian cubic for $6*6$ skew-symmetric matrices.
\end{itemize}

In each case above, one can provide an easy ad-hoc construction of a rank $n_{\A} = \dim_{\C} \A$ tautological vector bundle on $X_{\A}$. We shall rather give a uniform construction of this vector bundle which endows it with a global faithful right action of $\A$. For all $A = \begin{pmatrix} \lambda_1 & c & \overline{b} \\ \overline{c} & \lambda_2 & a \\ b & \overline{a} & \l3 \end{pmatrix} \in \J_{3}(\A)$, we denote by $\mathrm{Com}(A)$, the element of $\J_3(\A)$ defined by:
\[ \mathrm{Com}(A) = \begin{pmatrix} \lambda_2 \l3 - a \overline{a} & (\overline{b})(\overline{a})-\l3c & ca - \lambda_2\overline{b} \\ ab - \l3 \overline{c} & \lambda_1 \l3 -b \overline{b} & (\overline{c}) (\overline{b}) - \lambda_1 a \\ (\overline{a})(\overline{c})-\lambda_2b & bc - \lambda_1 \overline{a} & \lambda_1 \lambda_2 - c\overline{c}  \end{pmatrix}.\]
Owing to the associativity of $\A$ and to fundamental relation $\mathrm{Re}(cab) = \mathrm{Re}(abc) = \mathrm{Re}(bca)$, for any $c, b, a \in \A$, we get, for any $A \in \J_{3}(\A)$, the matrix factorization:
\begin{equation} \label{matfact}
\mathrm{Com}(A) \times A = A \times \mathrm{Com}(A) = \det_{\A}(A).I_3,
\end{equation}
where $I_3$ is the identity $3\times 3$ matrix. Furthermore, there is a natural embedding of algebras:
\[ L : \A \barrow \mathcal{M}_{n_{\A},n_{\A}}(\C),\]
sending an element to its left multiplication table and which induces an embedding of algebras:

\[ L : \J_{3}(\A) \barrow \mathcal{M}_{3n_{\A},3n_{\A}}(\C).\]
We fix $e_1,\ldots, e_{n_{\A}}$ a basis of $\A$ over $\C$ and consider the ring $\mathbb{C}[x_1,\ldots,x_{3n_{\A}+3}]$. We put $c = x_1e_1+ \ldots + x_ne_{n_{\A}}$, $b = x_{n_{\A}+1}e_1 + \ldots + x_{2n_{\A}}e_{n_{\A}}$, $a =x_{2n_{\A}+1}e_1 + \ldots + x_{3n_{\A}}e_{n_{\A}}$, $\lambda_1 = x_{3n_{\A}+1}e_1$, $\lambda_2 = x_{3n_{\A}+2}e_1$ and $\l3 = x_{3n_{\A}+3}e_1$. Consider the $3n_{\A} \times 3n_{\A}$ matrix with entries in $\mathbb{C}[x_1, \ldots, x_{3n_{\A}+3}]$ described in $n_{\A} \times n_{\A}$ blocks by:
\[M_{\A} := L(A) = \begin{pmatrix} \lambda_1.I_{n_{\A}} & L_{c} & {}^{t} L_{b} \\ {}^{t} L_{c} & \lambda_2.I_{n_{\A}}  & L_{a} \\ L_{b} & {}^{t} L_{a}& \l3.I_{n_{\A}}\end{pmatrix}.\]
One immediately notices that $M_{\A}$ is a symmetric $3n_{\A} \times 3n_{\A}$ matrix with entries in the ring $\mathbb{C}[x_1,\ldots,x_{3n_{\A}+3}]$. We also consider the matrix:

\[ \mathrm{Com}(M_{\A}) := L(\mathrm{Com}(A)) =  \begin{pmatrix} L_{\lambda_2 \l3 - a \overline{a}} & L_{(\overline{b})(\overline{a})-\l3c} & L_{ca - \lambda_2\overline{b}} \\ L_{ab - \l3 \overline{c}} & L_{\lambda_1 \l3 -b \overline{b}} & L_{(\overline{c}) (\overline{b}) - \lambda_1 a} \\ L_{(\overline{a})(\overline{c})-\lambda_2b} & L_{bc - \lambda_1 \overline{a}} & L_{\lambda_1 \lambda_2 - c\overline{c}}  \end{pmatrix}.\]
The matrix $\mathrm{Com}(M_{\A})$ is also a symmetric $3n_{\A} \times 3n_{\A}$ symmetric matrix with coefficients in $\mathbb{C}[x_1,\ldots,x_{3n_{\A}+3}]$. The matrix factorizations in equation (\ref{matfact}) and the fact that $L$ is a morphism of algebras show that we have a matrix factorizations in $\mathcal{M}_{3n_{\A},3n_{\A}}(\mathbb{C}[x_1,\ldots,x_{3n_{\A}+3}]$:
\begin{equation} \label{matfact2}
\mathrm{Com}(M_{\A}) \times M_{\A} = M_{\A} \times \mathrm{Com}(M_{\A}) = \mathrm{Det}_{\A}.I_{3n_{\A}}, 
\end{equation}
where $\mathrm{Det}_{\A}$ is the cubic polynomial in $\mathbb{C}[x_1,\ldots,x_{3n_{\A}+3}]$ corresponding to $\det_{\A}$. We deduce from Equation (\ref{matfact2}) that the degree $3 n_{\A}$ polynomial $\det_{\mathcal{M}_{3n_{\A},3n_{\A}}}(M_{\A})$ and the cubic polynomial $\mathrm{Det}_{\A}$ define set-theoretically the same hypersurface in $\mathbb{P}^{3n_{\A}+2}$ (that is the locus where $M_{\A}$ is not full-rank). On the other hand, it is known that $\mathrm{Det}_{\A}$ is irreducible (see the appendix of \cite{roberts}), so that we get the relation:
\begin{equation} 
\label{rank}
\det_{\mathcal{M}_{3n_{\A},3n_{\A}}}(M_{\A})  = (\mathrm{Det}_{\A})^{n_{\A}}.
\end{equation}

Let us denote by $E_{\A}$ the cokernel of:
\[ \A^{\oplus 3} \otimes_{\C} \OO_{\mathbb{P}(\J_{3}(\A))}(-1) \stackrel{M_{\A}}\barrow \A^{\oplus 3} \otimes_{\C} \OO_{\mathbb{P}(\J_{3}(\A))}\]
The sheaf $E_{\A}$ is supported on $\CC_{\A}$, as can be read from equation \ref{matfact2}. It is generically of rank $n_{\A}$ as follows from equation \ref{rank}. Since $\A$ is associative, we have:
$$ M_{\A} \times \begin{pmatrix} x \times a \\ y \times a \\ z \times a \end{pmatrix} = \left(M_{\A} \times \begin{pmatrix} x  \\ y  \\ z   \end{pmatrix} \right) \times a,$$
for all $\begin{pmatrix} x \\ y   \\ z   \end{pmatrix} \in \A^{\oplus 3}$ and all $a \in \A$. This implies that $E_{\A}$ is induced with a global faithful right action of $\A$. The sheaf $E_{\A}$ is a tautological sheaf on $\CC_{\A}$. Indeed, for any $A$ in the smooth locus of $\CC_{\A}$, the stalk $E_{\A}|_{A}$ is the cokernel of $A$. In order to transport this sheaf on $X_{\A}$, we recall the conormal construction. We denote by $N_{X_{\A}}$ the normal bundle of $X_{\A}$ in $\mathbb{P}(\J_{3}(\A))$. We have the conormal diagram:
\begin{equation*}
\xymatrix{ & &  \ar[lldd]_{p_{\A}} \mathbb{P}(N_{X_{\A}}^*(1)) \subset \mathbb{P}(\J_3(\A)) \times \mathbb{P}(\J_3(\A)^*) \ar[rrdd]^{q_{\A}} & &  \\
& & & & \\
X_{\A} \subset \mathbb{P}(\J_3(\A))  & & & & X_{\A}^* \subset \mathbb{P}(\J_3(\A)^*)}
\end{equation*}
where $X_{A}^* \subset \mathbb{P}(\J_3(\A)^*)$ is the projective dual of $X_{\A}$. It happens that $X_{A}^* \simeq \CC_{\A}$ (\cite{roberts}) under an identification $\mathbb{P}(\J_3(\A)) \simeq \mathbb{P}(\J_3(\A)^*)$ and that $E_{\A} = (q_{\A})_* \big(p_{\A}^* F_{\A} \big)$, where $F_{\A}$ is a rank $n_{\A}$ vector bundle on $X_{\A}$. The bundle $F_{\A}$ is the tautological bundle we were looking for. Indeed, for any $A \in X_{A}$, the stalk $F_{\A}|_{A}$ is the orthogonal in $(\A^{\oplus 3})^*$ of the kernel of $A$ (which is a $2 n_{\A}$-dimensional complex vector sub-space of $\A^{\oplus 3}$ as all matrices $A \in X_{\A}$ have rank $n_{\A}$). For the convenience of the reader, we describe explicitly $F_{\A}$ for each $\A$:
\begin{itemize}
\item if $\A = \R \otimes \C$, then $F_{\A} = \OO_{\mathbb{P}^2}(1)$,
\item  if $\A = \C \otimes \C$, then $F_{\A} = \OO_{\mathbb{P}^2 \times \mathbb{P}^2}(1,0) \oplus \OO_{\mathbb{P}^2 \times \mathbb{P}^2}(0,1)$,
\item if $\A = \H \otimes \C$, then $F_{\A} = (\mathcal{R}^*)^{\oplus 2}$, where $\mathcal{R}$ is the tautological bundle on $\mathrm{Gr}(2,6)$.

\end{itemize}

In the case of the octonions, the determinant $\det_{\O}(A)$ still makes sense for any $A \in \J_{3}(\O)$ and the equation $\det_{\O} = 0$ defines a cubic in $\mathbb{P}(\J_3(\O))$ known as the Cartan cubic. Its singular locus is the octonionic projective plane $\mathbb{OP}^2 \subset \mathbb{P}(\J_{3}(\O))$, a projective homogeneous space under $\mathrm{E}_6$ (see \cite{roberts}). There are however many key steps where the above construction of $F_{\A}$ fails for $\A = \O$. First of all, the matrix factorization:
$$\mathrm{Com}(A) \times A = A \times \mathrm{Com}(A) = \det_{\O}(A).I_3,$$
does not hold for all $A \in J_{3}(\O)$, due to the lack of associativity of $\O$. Furthermore, the embedding $L : \O \barrow \mathcal{M}_{8,8}(\C)$ is not an algebra embedding (because $\O$ is not associative). Hence, even if for all $A \in \J_{3}(\O)$ we could find $B_{A} \in \J_{3}(\O)$ such that $B_A \times A = A \times B_A = \det_{\O}(A).I_3$, there wouldn't necessarily be any corresponding matrix factorization involving $L(A)$.
\bigskip

Finally, the very notion of a locally free $\O \otimes \OO_{X}$-module on a scheme $X$ is a notoriously challenging one. While alternative modules over $\O$ have been studied by many authors (see for instance \cite{jacobson}), there are serious difficulties to sheafify these modules. Indeed, let $X$ be a scheme over $\mathbb{C}$ and $F$ be a sheaf of locally free (right) $\O \otimes \mathcal{O}_{X}$-modules of rank $n$ over $\O \otimes \OO_{X}$. Let $U,V$ trivializing opens in $X$. The transition map $t_{U,V}$ gives a right $\O$-linear isomorphism:

\begin{equation*}
\begin{split}
t_{U,V} \ : &  \ \mathbb{O} \otimes \C^n \times \left\{ U \cap V \right\} \stackrel{\sim}\barrow \mathbb{O} \otimes \C^n \times \left\{ U \cap V \right\} \\
           &   \ (a \otimes v ,x) \barrow (t_{U,V}(x)(a \otimes v),x).  
\end{split}
\end{equation*} 
The algebra $\O$ being non-associative, the only right $\O$-linear endomorphisms of $\O$ are given by multiplication by pure (complexified) real numbers in $\O$. We thus find that the transition map $t_{U,V}$ is of the form:

\begin{equation*}
\begin{split}
t_{U,V} \ : &  \ \mathbb{O} \otimes \C^n \times \left\{ U \cap V \right\} \stackrel{\sim}\barrow \mathbb{O} \otimes \C^n \times \left\{ U \cap V \right\} \\
           &   \ (a \otimes v ,x) \barrow (a \otimes r_{U,V}(x)(v),x),  
\end{split}
\end{equation*} 
for some cocyle $r \in H^{1}(X, \mathrm{GL}_n)$. As a consequence $F = \O \otimes E$ where $E$ is a vector bundle on $X$, and the global $\O$ structure on $F$ is trivial. This is in sharp contrast with the fact that the vector bundle $F_{\A}$ described above is a locally free $\A \otimes \OO_{X_{\A}}$-module of rank $n_{\A}$ (over $\OO_{X_{\A}}$) whose global right $\A$-structure is not trivial.

\bigskip

The above analysis suggests that the proof we give that the cokernel of the matrix $M_{\A}$ is a tautological rank $n_{\A}$ bundle on $\CC_{\A}$ endowed with a right global $\A$ action can not be adapted to the case $\A = \O$. But it doesn't actually show that the cokernel of $M_{\O}$ is the not, by some miracle, the bundle which existence is claimed by Bhargava and Ho. As a matter of fact, it happens that this cokernel doesn't even live on $\CC_{\O}$ and our first main result describes some of its geometric and representation theoretic properties. Let $c = x_1e_1+ \ldots + x_8e_8$, $b = x_{9}e_1 + \ldots + x_{16}e_n$, $a =x_{17}e_1 + \ldots + x_{24}e_n$, $\lambda_1 = x_{25}e_1$, $\lambda_2 = x_{26}e_1$ and $\l3 = x_{27}e_1$. Consider the $24 \times 24$ matrix with entries in $\mathbb{C}[x_1, \ldots, x_{27}]$ given in $8 \times 8$ blocks by:
\[M_{\O} = \begin{pmatrix} \lambda_1.I_8 & L_{c} &  {}^{t} L_{b} \\ {}^{t} L_{c} & \lambda_2.I_8 & L_{a} \\ L_{b} & {}^{t}L_{a} & \lambda_1.I_8 \end{pmatrix}.\]
Let $\mathrm{S}_{\mathrm{ODM}} $ be the sextic polynomial in $\mathbb{C}[x_1, \ldots, x_{27}]$ defined by:
 \[ \mathrm{S}_{\mathrm{ODM}}= \mathrm{Det}_{\O}^2 - 4 \phi(c,b,a)\mathrm{Det}_{\O} - \Big|[c,b,a] \Big|^2, \]
 where $\phi(c,b,a) = \frac{1}{2} \mathrm{Re}((c \overline{b} - \overline{b}c)a)$, $[c,b,a] = (cb)a-c(ba)$ and $\mathrm{Det}_{\O}$ is the cubic polynomial corresponding to $\det_{\O}$.

\begin{theo*} \label{theo1}

\begin{enumerate}
\item The vanishing locus of $\mathrm{S}_{\mathrm{ODM}}$ is an irreducible sextic (equally denoted by $\mathrm{S}_{\mathrm{ODM}}$) in $\mathbb{P}(\J_{3}(\O))$ which singular locus has codimension $4$ in $\mathbb{P}(\J_{3}(\O))$. Furthermore, up to a finite quotient, we have:
$$\mathrm{Aut}^{0}(\mathrm{S}_{\mathrm{ODM}}) = \mathrm{SL}_{3}(\C) \times \mathrm{SO}_7(\C).$$

\item The cokernel of:
$$ \O^{\oplus 3} \otimes_{\C} \OO_{\mathbb{P}(\J_{3}(\O))}(-1) \stackrel{M_{\O}}\barrow \O^{\oplus 3} \otimes_{\C} \OO_{\mathbb{P}(\J_{3}(\O))}$$
is supported on $\mathrm{S}_{\mathrm{ODM}}$. It has generically rank $4$ and is $ \mathrm{SL}_{3}(\C) \times \mathrm{SO}_7(\C)$-equivariant.

\end{enumerate}
 
\end{theo*}

The sextic in Theorem \ref{theo1} has been first found out in \cite{O} and was then rediscovered later in \cite{DM} with computer aid. We call it the Ogievetski\^i-Dray-Manogue sextic and denote it $\mathrm{S}_{\mathrm{ODM}}$. Theorem \ref{theo1} will be proved in section in the setting of the octonionic eigenvalue problem.
As far as I know, the objects introduced in Theorem \ref{theo1} have not been yet investigated from the algebraic geometry perspective and they certainly deserve to be studied in details. They seem however not directly related to the geometry of the Cartan cubic and the tautological bundles that might exist on it. Our second main result relates to this more classical object. We denote by $N_{\O}$ the $24 \times 24$ matrix with entries in $\mathbb{C}[x_1, \ldots, x_{27}]$ given in $8 \times 8$ blocks by:
\[N_{\O} = \begin{pmatrix} \lambda_1.I_8 & R_{c} &  {}^{t} L_{b} \\ {}^{t} R_{c} & \lambda_2.I_8 & L_{a} \\ L_{b} & {}^{t}L_{a} & \lambda_1.I_8 \end{pmatrix},\]
where $R_c$ is the matrix of right multiplication by $c \in \O$.

\begin{theo*} \label{theo2}
\begin{enumerate}
\item Let $E_{\O}$ be the cokernel of
$$ \O^{\oplus 3} \otimes_{\C} \OO_{\mathbb{P}(\J_{3}(\O))}(-1) \stackrel{N_{\O}}\barrow \O^{\oplus 3} \otimes_{\C} \OO_{\mathbb{P}(\J_{3}(\O))}.$$
The sheaf $E_{\O}$ is supported on a degree $9$ reducible hypersurface in $\mathbb{P}(\J_{3}(\O))$, which is the union of hypersurface projectively isomorphic to $\CC_{\O}$ (the Cartan cubic) and a sextic hypersurface, which we denote $\mathfrak{S}$, and whose equation is:
\begin{equation*}
\begin{split}
\mathfrak{S} & = \left(\ll_1|a|^2 + \ll_2|b|^2 + \ll_3|c|^2 - \ll_1\ll_2\ll_3 \right)^2 \\
                              & \hspace{0.5cm} - 4\left(\ll_1|a|^2 + \ll_2|b|^2 + \ll_3|c|^2 - \ll_1\ll_2\ll_3 \right)\mathrm{Re}(c)\mathrm{Re}(ab) \\
                              & \hspace{0.5cm} + 4\left(|b|^2|a|^2\mathrm{Re}(c)^2 + |c|^2 \mathrm{Re}(ab)^2 - |a|^2|b|^2|c|^2 \right),
\end{split}
\end{equation*}
where $|z|^2 = z\overline{z}$  and $\mathrm{Re}(z) = \frac{z + \overline{z}}{2}$.

\item The cubic hypersurface is $\langle \mathrm{Spin}_7(\C), \mathrm{SL}_3(\C) \rangle$-invariant, where $ \langle \mathrm{Spin}_7(\C), \mathrm{SL}_3(\C) \rangle$ is the subgroup of $\mathrm{SL}_ {24}$ generated by $\mathrm{Spin}_7(\C)$ and $\mathrm{SL}_3(\C)$. The restriction of $E_{\O}$ to the cubic has generically rank $4$ and is $\langle \mathrm{Spin}_7(\C), \mathrm{SL}_3(\C) \rangle$-equivariant.
\item The hypersurface $\mathfrak{S}$ is equally $\langle \mathrm{Spin}_7(\C), \mathrm{SL}_3(\C) \rangle$-invariant. The restriction of $E_{\O}$ to $\mathfrak{S}$ has generically rank $2$ and is $\langle \mathrm{Spin}_7(\C), \mathrm{SL}_3(\C) \rangle$-equivariant.
\end{enumerate}
\end{theo*}
It should be noted that the explicit equation of $\mathfrak{S}$ has been first found by Jonathan Hauenstein using the \textit{Bertini} software. The embedding of $\mathrm{Spin}_7$ in $\mathrm{SL}_{24}$ we consider does not factor through $\mathrm{E}_6$. As a matter of fact, we haven't been able to identify the subgroup of $\mathrm{SL}_ {24}$ generated by $\mathrm{Spin}_7(\C)$ and $\mathrm{SL}_3(\C)$. We can however show the:

\begin{theo*} \label{theo3}
Let $G$ be the subgroup of $\mathrm{SL}_ {24}$ generated by $\mathrm{Spin}_7(\C)$ and $\mathrm{SL}_3(\C)$. The action of $\mathbb{C}^* \times G$ on $\J_{3}(\O)$ is prehomogeneous.
\end{theo*}
Theorem \ref{theo3} is slightly surprising as the subgroup $G$ of $\mathrm{SL}_{24}$ generated by $\mathrm{Spin}_7(\C)$ and $\mathrm{SL}_3(\C)$ is \textbf{not} $\mathrm{E}_6$ (and is even not a subgroup of $\mathrm{E}_6$). We will prove this fact in details in section 3.3, but less us have a glimpse at the argument. Any element of $\J_{3}(\O)$ can be diagonalized under the action of $\mathrm{E}_6$. Hence, any vector of $\J_{3}(\O)$ can be transformed, under the action of $\mathrm{E}_6$, to a vector whose corresponding $24*24$ matrix has rank $0,8,16$ or $24$. On the other hand, the action of $G$ on $\J_{3}(\O)$ preserves the rank. Since there are vectors in $\J_{3}(\O)$ whose corresponding $24*24$ matrix have rank $22$, we conclude that $G$ is not $\mathrm{E}_6$. The prehomogeneous space $(\J_{3}(\O), \mathbb{C}^* \times G)$ does not appear in the list of Kimura-Sato \cite{kimura}. This is not surprising as it is most likely not irreducible as a prohomogeneous space. It thus provides an interesting example of a non-irreducible prehomogeneous vector space which does not split as a sum of lower dimensional prehomogeneous vector spaces.

\bigskip

We prove Theorems \ref{theo2} and \ref{theo3} in the second section of this paper. There, we introduce a variant of the octonionic eigenvalue problem, which we call \textit{the twisted octonionic eigenvalue problem}. We deduce the aforementioned results from both algebraic and geometric features of this twisted problem. We note that the geometry of the twisted octonionic eigenvalue problem differs strongly from that of the corresponding untwisted problem. Indeed, one easily notices that the action of $\mathrm{GL}_3 \times \mathrm{SO}_7$ on $\J_{3}(\O)$ is far from being homogeneous (the generic orbit has codimension at least $6$ in $\J_{3}(\O)$). Theorem \ref{theo3} thus shows that the twisted problem is eminently more symmetric than the untwisted problem. 

\bigskip

The sheaf $E_{\O}$ appearing in Theorem \ref{theo2} may be considered as a quasi-tautological sheaf on $\CC_{\O}$ as it is the cokernel of a matrix whose tautological interpretation is clear. However, the description of $E_{\O}|_{A}$ for generic $A \in \CC_{\O}$ is not completely obvious. Furthermore, the restriction of $E_{\O}$ to the smooth locus of $\CC_{\O}$ is not a vector bundle, so that $E_{\O}$ is not the Fourier-Mukai transform, by the conormal diagram, of a vector bundle on $\mathbb{OP}^2$. One can however expect that it is the Fourier-Mukai transform of a rank $4$ coherent sheaf on $\mathbb{OP}^2$ which could be conceived as quasi-tautological sheaf on $\mathbb{OP}^2$. The tautological interpretation of this sheaf is again not completely transparent. Further investigations should reveal many more fascinating features associated to the twisted octonionic eigenvalue problem. In a forthcoming paper with Jonathan Hauenstein \cite{Jonme}, we will we study in more details the geometry and singularities of $\SS$ and $\mathfrak{S}$.

\bigskip

\noindent \textbf{Acknowledgements :} I am very grateful to Jonathan Hauenstein for the many discussions we have had on the geometry of $\SS$ and $\mathfrak{S}$ and especially for providing me with the explicit equation of $\mathfrak{S}$ using the software \textit{Bertini}. Many thanks to Oleg Ogievetski\^i for sharing with me an English translation of his paper \cite{O}. I would also like to acknowledge discussions I have had with Robert Bryant and I thank him for the clean argument he communicated to me showing that the subgroup of $\mathrm{E}_6$ generated by $\mathrm{Spin_7}$ and $\mathrm{SL}_3$ is $\mathrm{E}_6$ itself.

\end{section}

\begin{section}{The eigenvalue problem for $\J_{3}(\O)$ and the Ogievetski\^i-Dray-Manogue sextic}

\begin{subsection}{The octonionic eigenvalue problem}

Let $\A$ be $\mathbb{C}$, $\C \otimes_{\R} \C$, $\mathbb{H}$ or $\O$ (the algebra $\mathbb{H}$ and $\O$ being the \textit{complexified} quaternions and octonions). The left (complex) eigenvalue problem for $\J_{3}(\A)$ is the following. Given $A = \begin{pmatrix} \lambda_1 & c & \overline{b} \\ \overline{c} & \lambda_2 & a \\ b & \overline{a} & \l3 \end{pmatrix} \in \J_{3}(\A)$, what are the $\mu \in \mathbb{C}$ such that there exists a non zero vector $\begin{pmatrix} x \\ y \\ z \end{pmatrix} \in \A^{3}$ with $A \begin{pmatrix} x \\ y \\ z \end{pmatrix} = \mu . \begin{pmatrix} x \\ y \\ z \end{pmatrix}$? So that there is no confusion due to the non-commutativity (and possibly non associativity) of $\A$, the equality $A \begin{pmatrix} x \\ y \\ z \end{pmatrix} = \mu . \begin{pmatrix} x \\ y \\ z \end{pmatrix}$ is meant to be:
\begin{equation*}
\left\{
\begin{split}
\lambda_1 x + c y + \overline{a}_2 z & \ = \mu x \\
 \overline{c} x + \lambda_2 y +  a z & \ = \mu y \\
b x + \overline{a} y + \lambda_2 z & \ = \mu z. \\
\end{split}
\right.
\end{equation*}
Let us say that $A \in \J_{3}(\A)$ is $\textit{singular}$ or $\textit{degenerate}$ if $0$ is an eigenvalue of $A$. In case $\A$ is associative, the following proposition is well-known. We recall a proof for the convenience of the reader.

\begin{prop}
Let $\A$ be an associative composition algebra which is finite dimensional over $\C$. Let $\begin{pmatrix} \lambda_1 & c & \overline{b} \\ \overline{c} & \lambda_2 & a \\ b & \overline{a} & \l3 \end{pmatrix} \in \J_{3}(\A)$ and $\mu \in \mathbb{C}$. Then, there  exists a non zero vector $\begin{pmatrix} x \\ y \\ z \end{pmatrix} \in \A^{3}$ with $A \begin{pmatrix} x \\ y \\ z \end{pmatrix} = \mu . \begin{pmatrix} x \\ y \\ z \end{pmatrix}$ if and only if $\det_{\A}(A - \mu.I_3)=0$.
\end{prop}

\begin{proof}
If $A = 0$ the result is obvious. We now assume that $A \neq 0$, for instance $c \neq 0$ (the non-vanishing of any other coefficient in $A$ is dealt with exactly in the same way). Assume that $\det_{\A}(A- \mu.I_3) = 0$. If $\mathrm{Com}(A - \mu.I_3) \neq 0$ then, the equation:
\[ (A-\mu.I_3) \times \mathrm{Com}(A-\mu.I_3) = \det_{\A}(A-\mu.I_3).I_3 = 0\]
shows that any non-zero column vector of $\mathrm{Com}(A-\mu.I_3)$ is an eigenvector for $A$ with respect to $\mu$. If $\mathrm{Com}(A-\mu.I_3) = 0$, then one easily checks that $\begin{pmatrix} \lambda_2-\mu \\ -\overline{c} \\ 0 \end{pmatrix}$ is a non-zero eigenvector for $A$ with respect to $\mu$.

\bigskip

\noindent Assume on the other side that there exists $\begin{pmatrix} x \\ y \\ z \end{pmatrix} \neq 0$ such that $(A- \mu.I_3) \begin{pmatrix} x \\ y \\ z \end{pmatrix} = 0$. Let $B = \begin{pmatrix} x & x &x \\ y&y&y \\ z&z&z \end{pmatrix}$. We have $(A-\mu.I_3) \times B = 0$ and since $\A$ is associative, we also have $\mathrm{Com}(A-\mu.I_3) \times (\A-\mu.I_3) \times B = 0$. But $\mathrm{Com}(A-\mu.I_3) \times (A-\mu.I_3) = \det_{\A}(A-\mu.I_3) . I_3$, so that we get:
\[ \det_{\A}(A-\mu.I_3) . B = 0.\]
The matrix $B$ is non-zero by hypothesis and $\det_{\A}(A -\mu.I_3) \in \C$ imply that $\det_{\A}(A - \mu.I_3)=0$.
\end{proof}

\noindent The algebra $\O$ being not associative, the above result does not apply in this case. Nevertheless, Ogievetski\^i found out a characteristic equation that determines the complex eigenvalues of $A \in \J_{3}(\O)$:

\begin{theo}[see \cite{O, DM}] \label{prop5}
\begin{enumerate}
\item Let $A = \begin{pmatrix} \lambda_1 & c & \overline{b} \\ \overline{c} & \lambda_2 & a \\ b & \overline{a} & \l3 \end{pmatrix} \in \J_{3}(\O)$. The matrix $A$ is singular if and only if the $24*24$ matrix $\begin{pmatrix} \lambda_1.I_8 & L_c & {}^{t} L_b \\ {}^{t} L_c & \lambda_2.I_8 & L_a \\ L_b & {}^{t} L_a & \l3.I_8 \end{pmatrix}$ is singular. 

\item The degeneracy locus of the matrix $\begin{pmatrix} \lambda_1.I_8 & L_c & {}^{t} L_b \\ {}^{t} L_c & \lambda_2.I_8 & L_a \\ L_b & {}^{t} L_a & \l3.I_8 \end{pmatrix}$ is given by the equation:
\[  (\det_{\O}(A))^2 -  4 \phi(c,b,a) \det_{\O}(A) - \Big|[c,b,a] \Big|^2  =0.\]
\end{enumerate}
\end{theo}
As we will build upon it when dealing with the twisted version of the octonionic eigenvalue problem, we recall the main steps of the proof from \cite{O}:

\begin{proof}
The first assertion in the Theorem is obvious and follows from the definitions. We prove the second assertion. For $a,b,c \in \O$ and $\ll_1, \ll_2, \ll_3 \in \mathbb{C}$ such that $|a|,|b|,|c|, \ll_1, \ll_2, \ll_3 \neq 0$ and $|c|^2\ll_1 - \dfrac{|c|^4}{\ll_2} \neq 0$, we define the following $24*24$ matrices in block form:
\[ R_1 = \begin{pmatrix} \frac{1}{|c|^2} L_c & 0_{8,8} & 0_{8,8} \\ 0_{8,8}& I_{8}& 0_{8,8} \\ 0_{8,8} & 0_{8,8}& \frac{1}{|a|^2} {}^{t} L_a \end{pmatrix} \ \ \textrm{and} \ \   R_2 = \begin{pmatrix} I_8 & \frac{|c|^2}{\ll_2} & 0_{8,8} \\ 0_{8,8}& I_{8}& 0_{8,8} \\ \frac{1}{\delta} B & \frac{|a|^2}{\ll_2}I_8 & I_8 \end{pmatrix},\]
where $B = L_aL_bL_c - \dfrac{|a|^2|c|^2}{\ll_2}.I_8$, $\delta = |c|^2\ll_1 - \dfrac{|c|^4}{\ll_2}$, $0_{8,8}$ is the zero $8*8$ matrix and $I_8$ the identity matrix of size $8$.
One easily checks that:

\[ \begin{pmatrix} \lambda_1.I_8 & L_c & {}^{t} L_b \\ {}^{t} L_c & \lambda_2.I_8 & L_a \\ L_b & {}^{t} L_a & \l3.I_8 \end{pmatrix} = R_1R_2 \begin{pmatrix} \delta.I_8 & 0_{8,8} & 0_{8,8} \\ 0_{8,8} & \ll_2.I_8 & 0_{8,8} \\ 0_{8,8} & 0_{8,8} & \nu.I_8 - \dfrac{1}{\delta} B{}^{t}B \end{pmatrix} {}^{t} R_2 {}^{t} R_1, \]
where $\nu = \ll_3|a|^2 - \frac{|a|^4}{\ll_2}$. Since $\det(R_1) =\dfrac{1}{|a|^8|c|^8}$ and $ \det(R_2) = 1$, we get:
\begin{equation*}
\begin{split}
\det \begin{pmatrix} \lambda_1.I_8 & L_c & {}^{t} L_b \\ {}^{t} L_c & \lambda_2.I_8 & L_a \\ L_b & {}^{t} L_a & \l3.I_8 \end{pmatrix} & = \dfrac{1}{|a|^{16} |c|^{16}} \det \left( \delta \nu \ll_2.I_8 - \ll_2.B{}^{t} B \right)\\
                                                                    & = \det \left((\ll_1 \ll_2 \ll_3 - \ll_1|a|^2 - \ll_2 |b|^2 - \ll_3 |c|^2).I_8 + (L_aL_bL_c + {}^{t}L_c {}^tL_b {}^{t} L_a) \right)\\
                                                                    & = \det( -( 2 \mathrm{Re}(cab) - \det_{\O}(A)).I_8 + (L_aL_bL_c + {}^{t}L_c {}^tL_b {}^{t} L_a)),
\end{split}
\end{equation*}
where $A = \begin{pmatrix} \lambda_1 & c & \overline{b} \\ \overline{c} & \lambda_2 & a \\ b & \overline{a} & \l3 \end{pmatrix}.$

The expression $\det(-(2 \mathrm{Re}(cab) - \det_{\O}(A)).I_8 + (L_aL_bL_c + {}^{t}L_c {}^tL_b {}^{t} L_a))$ has no denominators so that the above equality holds for all  $a,b,c \in \O$ and $\ll_1, \ll_2, \ll_3 \in \mathbb{C}$. We deduce the following:

\smallskip

\noindent The determinant of $\begin{pmatrix} \lambda_1.I_8 & L_c & {}^{t} L_b \\ {}^{t} L_c & \lambda_2.I_8 & L_a \\ L_b & {}^{t} L_a & \ll_3.I_8 \end{pmatrix}$ vanishes if and only if $ 2 \mathrm{Re}(cab) - \det_{\O}(A)$ is a (complex) eigenvalue of $(L_aL_bL_c + {}^{t}L_c {}^tL_b {}^{t} L_a)$. We shall now determine the (complex) eigenvalues of $(L_aL_bL_c + {}^{t}L_c {}^tL_b {}^{t} L_a)$. This means we are looking for the $\k \in \C$ such that:
\begin{equation} \label{Ogiet}
\exists x \in \O, \ x \neq 0 \ \ a.(b.(c.x)) + \overline{c}.(\overline{b}.(\overline{a}.x)) = \kappa.x
\end{equation}
Let us consider a four dimensional quaternionic subalgebra of $\O$ which contains $a$ and $c$ (its existence is granted by a Theorem of Hurwitz). Let $e \in \O$ be a vector orthogonal to $M$ such that $\O = M \oplus M.e$. We write $b = b_0 + b_1.e$ and $x = x_0 + x_1.e$, with $b_0,b_1,x_0,x_1 \in M$. Using the relations:
$$ u.(v.e) = (v.u).e, \ (u.e).(v.e) = -\overline{v}.u, \ u.e = e. \overline{u} \ \textrm{and} \ (u.e).v = (u. \overline{v}).e,$$
valid for all $(u,v) \in M^2$, we find that equation \ref{Ogiet} is equivalent to the system:
\begin{equation*}
\left\{
\begin{split}
2 \mathrm{Re}(c.a.b_0).x_0 + (\overline{c}.a - a.\overline{c}).\overline{x_1}.b_1 & = \kappa.x_0 \\
(\overline{a}.c - c.\overline{a}).x_0.\overline{b_1} + 2\mathrm{Re}(c.b_0.a).\overline{x_1} & = \kappa.\overline{x_1}.
\end{split}
\right.
\end{equation*}
This system can be written in matrix form:
\[ \begin{pmatrix} (2\mathrm{Re}(c.a.b_0) - \kappa).I_4 & ({}^{t} L_c L_a - L_a{}^{t} L_c) R_{b_1} \\  ({}^{t} L_a L_c - L_c{}^{t} L_a) {}^{t}R_{b_1} & (2\mathrm{Re}(c.b_0.a) - \kappa).I_4 \end{pmatrix} \times \begin{pmatrix} x_0 \\ \overline{x_1} \end{pmatrix} = \begin{pmatrix} 0_{4,1} \\ 0_{4,1} \end{pmatrix}. \]
The existence of such an $\begin{pmatrix} x_0 \\ \overline{x_1} \end{pmatrix} \neq 0$ is then equivalent to:
$$ \det \begin{pmatrix} (2\mathrm{Re}(c.a.b_0) - \kappa).I_4 & ({}^{t} L_c L_a - L_a{}^{t} L_c) R_{b_1} \\  ({}^{t} L_a L_c - L_c{}^{t} L_a) {}^{t}R_{b_1} & (2\mathrm{Re}(c.b_0.a) - \kappa).I_4 \end{pmatrix} = 0.$$
The bottom right entry of the above block-matrix commutes with every matrix, so that the block-computation of the determinant gives:
\begin{equation*}
\begin{split}
\det & \begin{pmatrix} (2\mathrm{Re}(c.a.b_0) - \kappa).I_4 & ({}^{t} L_c L_a - L_a{}^{t} L_c) R_{b_1} \\  ({}^{t} L_a L_c - L_c{}^{t} L_a) {}^{t}R_{b_1} & (2\mathrm{Re}(c.b_0.a) - \kappa).I_4 \end{pmatrix} \\
& = \det( (2\mathrm{Re}(c.a.b_0) - \kappa)(2\mathrm{Re}(c.b_0.a) - \kappa).I_4 -  (({}^{t} L_c L_a - L_a{}^{t} L_c) R_{b_1}) \times (({}^{t} L_a L_c - L_c{}^{t} L_a) {}^{t}R_{b_1}))\\
& = \det( \left((2\mathrm{Re}(c.a.b_0) - \kappa)(2\mathrm{Re}(c.b_0.a) - \kappa) - |\overline{c}.a-a.\overline{c}|^2.|b_1|^2 \right).I_4).
\end{split}
\end{equation*}
We deduce that:
$$ \textrm{ $\kappa$ is an eigenvalue of $(L_aL_bL_c + {}^{t}L_c {}^tL_b {}^{t} L_a)\ \Leftrightarrow \ (2\mathrm{Re}(c.a.b_0) - \kappa)(2\mathrm{Re}(c.b_0.a) - \kappa) - |\overline{c}.a-a.\overline{c}|^2.|b_1|^2 = 0$.}$$
We note that $b = b_0 + b_1.e$ and $e \perp M$, hence:
\[ \mathrm{Re}(c.a.b_0) = \mathrm{Re}(c.a.b), \ \textrm{Re}(c.b_0.a) = \mathrm{Re}(c.b.a) \ \textrm{and} \ |\overline{c}.a-a.\overline{c}|^2.|b_1|^2 = |[a,b_1.e,c]|^2 = |[c,b,a]|^2.\]
We furthermore recall that $2\mathrm{Re}(c.a.b) - \det_{\O}(A)$ is an eigenvalue of $(L_aL_bL_c + {}^{t}L_c {}^tL_b {}^{t} L_a)$ if and only if the matrix $\begin{pmatrix} \lambda_1.I_8 & L_c & {}^{t} L_b \\ {}^{t} L_c & \lambda_2.I_8 & L_a \\ L_b & {}^{t} L_a & \ll_3.I_8 \end{pmatrix}$ is singular. We finally find that the degeneracy locus of the matrix $\begin{pmatrix} \lambda_1.I_8 & L_c & {}^{t} L_b \\ {}^{t} L_c & \lambda_2.I_8 & L_a \\ L_b & {}^{t} L_a & \ll_3.I_8 \end{pmatrix}$ is defined by the equation:
\begin{equation*}
\begin{split}
 & \ (2\mathrm{Re}(c.a.b) - \kappa)(2\mathrm{Re}(c.b.a) - \kappa) - |[c,b,a]|^2 = 0 \\
\Leftrightarrow \ & \ \mathrm{Det}_{\O}(2\mathrm{Re}(c.a.b) - 2\mathrm{Re}(c.b.a) + \mathrm{Det}_{\O}) - |[c,b,a]|^2 = 0\\
\Leftrightarrow \ & \ \mathrm{Det}_{\O}^2 + 2\mathrm{Re}(c(ab-ba)) \mathrm{Det}_{\O} -  |[c,b,a]|^2 = 0  \\
\Leftrightarrow \  &\ \mathrm{Det}_{\O}^2 + 4 \phi(a, \overline{b}, c) \mathrm{Det}_{\O} -  |[c,b,a]|^2 = 0 \\
\Leftrightarrow \ & \ \mathrm{Det}_{\O}^2 - 4 \phi(a, b, c) \mathrm{Det}_{\O} -  |[c,b,a]|^2 = 0.\\
\end{split}
\end{equation*}
\end{proof}
\end{subsection}
\begin{subsection}{The geometry of the Ogievetski\^i-Dray-Manogue sextic}
We will now describe elementary geometric features of the sextic $\SS \subset \mathbb{P}^{26}$ defined by the equation:
\[ (\det_{\O}(A))^2 -  4 \phi(c,b,a) \det_{\O}(A) - \Big|[c,b,a] \Big|^2  =0.\]
Some of them have been obtained with the help of \textit{Macaulay2} (\cite{M2}):

\begin{theo} \label{mopiii1}
\begin{enumerate}
\item The sextic $\SS \subset \mathbb{P}^{26}$ is irreducible and its singular locus is of codimension $4$ in $\mathbb{P}^{26}$. 

\item The neutral component the automorphism group of $\SS$ is $\mathrm{SO_7} \times \mathrm{SL}_{3}$.

\item The map:
\[ \mathrm{Gr}(6,\J_{3}(\O))/\!\!/ \mathrm{SO_7} \times \mathrm{SL}_{3} \barrow \mathrm{S}^6 \C^6 /\!\!/  \mathrm{GL}_{6}\]
which sends a six dimensional subspace $L \subset \J_{3}(\O)$ to the sextic $\SS \cap L$ is generically \'etale onto its image.
\end{enumerate}
\end{theo}

\begin{proof}
The codimension of the singular locus of $\SS \subset \mathbb{P}^{26}$ is obtained computationally with the help of \textit{Macaulay2}. The only <<challenge>> is to find a way to write explicitly the equation $\SS$ as, to the best of my knowledge, there is no obvious package to simulate the octonion algebra in \textit{Macaulay2}. We refer to the first appendix for a description of our algorithm. The irreducibility of $\SS$ follows as it is then a normal hypersurface in $\mathbb{P}^{26}$.
\bigskip

\noindent We now prove the second and third assertion. We first describe the action of $\mathrm{SO_7} \times \mathrm{SL}_{3}$ on $\J_{3}(\O)$. Let $T_1 \in \mathrm{SO_7}$ and $A =  \begin{pmatrix} \lambda_1 & c & \overline{b} \\ \overline{c} & \lambda_2 & a \\ b & \overline{a} & \l3 \end{pmatrix} \in \J_{3}(\O)$, we put:
\[ T_1.A = \begin{pmatrix} \lambda_1 & T_1(c) & (\overline{T_1(b)}) \\ (\overline{T_1(c)}) & \lambda_2 & T_1(a) \\ T_1(b) & (\overline{T_1(a)}) & \l3 \end{pmatrix} \in \J_{3}(\O),\]
where $\mathrm{SO_7} \subset \mathrm{SO_8}$ is the isotropy group of $1 \in \O$.
On the other hand, for $C \in \mathrm{SL_3}$ and $A =  \begin{pmatrix} \lambda_1 & c & \overline{b} \\ \overline{c} & \lambda_2 & a \\ b & \overline{a} & \l3 \end{pmatrix} \in \J_{3}(\O)$, we put:
\[ C.A = {}^{t} C \times A \times C \in \J_{3}(\O),\]
where $\times$ denotes here the product of $3*3$ matrices (there is no associativity issues since all coefficients of $C$ are complex and thus commute with all coefficients of $A$). Since $T_1(x) = x$, for all $T_1 \in \mathrm{SO_7}$ and for all $x \in \O$ with zero imaginary part, we deduce that:
\[ C.(T_1.A) = T_1.(C.A),\]
for all $T_1 \in \mathrm{SO_7},\ C \in \mathrm{SL_3}$ and $A \in \J_{3}(\O)$. Furthermore, we notice that $\mathrm{SO_7} \cap \mathrm{SL_3} = \{Id\}$ as subgroups of $\mathrm{SL}_{27}$. As a consequence, we find that the subgroup of $\mathrm{SL}_{27}$ generated by $\mathrm{SO_7}$ and $\mathrm{SL_3}$ is $\mathrm{SO_7} \times \mathrm{SL_3}$. Let $A \in \J_{3}(\O)$ be a singular matrix. By definition, there exists $\begin{pmatrix} x \\ y \\ z \end{pmatrix} \in \O^{\oplus 3}$, with $\begin{pmatrix} x \\ y \\ z \end{pmatrix} \neq 0_{\O^{\oplus 3}}$ such that:
\[A  \times \begin{pmatrix} x \\ y \\ z \end{pmatrix} = 0_{\O^{\oplus 3}}.\]
For any $T_1 \in \mathrm{SO_7}$, the \textit{Triality principle} (see section 2 of \cite{triality} for instance) insures the existence of $T_2 \in \mathrm{SO}_8$ such that:
\[\forall (u,v) \in \O^2, \  T_2(u.v) = T_1(u).T_2(v).\]
One then easily computes that:
\[(T_1.A) \times \begin{pmatrix} T_2(x) \\ T_2(y) \\ T_2(z) \end{pmatrix}  = \begin{pmatrix} T_2 & & \\ &T_2 & \\ & & T_2 \end{pmatrix} \times \left(A \times    \begin{pmatrix} x \\ y \\ z \end{pmatrix} \right) = 0_{\O^{\oplus 3}}.\]
We have $\begin{pmatrix} T_2(x) \\ T_2(y) \\ T_2(z) \end{pmatrix} \neq 0_{\O^{\oplus 3}}$, hence $T_1.A$ is also a singular matrix. Furthermore, we obviously note that for all $C \in \mathrm{SL_3}$, we have:
$$ (C.A) \times \left( C^{-1} \times \begin{pmatrix} x \\ y \\ z \end{pmatrix} \right)  = 0_{\O^{\oplus 3}},$$
which shows that $C.A$ is also a singular element of $\J_{3}(\O)$. The variety of singular matrices in $\J_{3}(\O)$ is thus preserved by the action of $\mathrm{SO_7} \times \mathrm{SL_3}$. By Theorem \ref{prop5}, we deduce that $\mathrm{SO_7} \times \mathrm{SL_3}$ is a subgroup of the neutral component of $\mathrm{Aut}(\SS)$.

\bigskip

In order to prove that $\mathrm{Aut}^{0}(\SS) = \mathrm{SO_7} \times \mathrm{SL_3}$ (up to a finite quotient), we show, using \textit{Macaulay2}, that the Lie algebra of $\mathrm{Aut}^{0}(\SS)$ has dimension less or equal to $29$. Consider the map:
\begin{equation*}
\begin{split}
\varphi \,\,\, : \,\,\, & M_{6 \times 27} /\!\!/ \mathrm{Aut}^{0}(\SS) \barrow S^6 \mathbb{C}^{6} \\
         & \left[(m)_{i,j} \right] \barrow \SS(m_{1,1}.z_1 + \cdots + m_{1,6}.z_6, \cdots, m_{27,1}.z_1 + \cdots + m_{27,6}.z_6),\\ 
\end{split}         
\end{equation*}
where $z_1,\cdots,z_6$ is a basis for $(\mathbb{C}^{6})^*$. Let us denote by $\mathfrak{aut}(\SS)$ the Lie algebra of $\mathrm{Aut}^{0}(\SS)$. The differential of $\phi$ at a point $\left[(m)_{i,j} \right] \in M_{6 \times 27}$ is given by :
 
\begin{equation*}
\begin{split}
d \varphi_{\left[m_{i,j}\right]} \,\,\, : \,\,\, & M_{6 \times 27} / \mathfrak{aut}(\SS) \barrow S^6 \mathbb{C}^{6} \\
         & \left[q_{i,j} \right] \barrow \sum_{i=1}^{27} \left( \sum_{j=1}^{6} q_{i,j}.z_j \right) \times \dfrac{\SS}{\partial x_i}(m_{1,1}.z_1 + \cdots + m_{1,6}.z_6, \cdots, m_{27,1}.z_1 + \cdots + m_{27,6}.z_6),\\
\end{split}
\end{equation*}
We compute with the help of \textit{Macaulay2} that the rank of $d \varphi$ at a generic point in $ M_{6 \times \J_{3}(\O)} / \mathfrak{aut}(\SS)$ is bigger or equal to $133$ (see the appendix A for the details of the algorithm). We deduce that:
$$\dim  M_{6 \times \J_{3}(\O)} / \mathfrak{aut}(\SS) \geq 133,$$ which implies that $\dim \mathfrak{aut}(\SS) \leq 29$. Since $\mathrm{SO_7} \times \mathrm{SL_3} \subset \mathrm{Aut}^{0}(\SS)$ and $\dim \mathrm{SO_7} \times \mathrm{SL_3} = 29$, we find that:

$$\mathrm{Aut}^{0}(\SS) = \mathrm{SO_7} \times \mathrm{SL_3}.$$

\bigskip

We can now infer that the map:
\[ \mathrm{Gr}(6,\J_{3}(\O))/\!\!/ \mathrm{SO_7} \times \mathrm{SL}_{3} \barrow \mathrm{S}^6 \C^6 /\!\!/  \mathrm{GL}_{6}\]
is generically \'etale onto its image. Indeed, the rank of $d \phi$ being generically $133$ and $\dim  M_{6 \times \J_{3}(\O)} /\!\!/ \mathrm{SO_7} \times \mathrm{SL}_{3} = 162 - 29 = 133$, we find that $\phi$ is generically \'etale onto its image. We finally notice that $\phi$ is $\mathrm{GL}_{6}$-equivariant, so that the map induces by $\varphi$ on the quotients by $\mathrm{GL}_6$:
\[  \big(M_{6 \times \J_{3}(\O)} /\!\!/ \mathrm{SO_7} \times \mathrm{SL}_{3}\big) /\!\!/ \mathrm{GL}_6 =  \mathrm{Gr}(6,\J_{3}(\O))/\!\!/ \mathrm{SO_7} \times \mathrm{SL}_{3}  \barrow \mathrm{S}^6 \C^6 /\!\!/  \mathrm{GL}_{6}\]
is equally generically \'etale onto its image.
\end{proof}

\begin{rem} \label{rem1}
The generic isotropy group for the action of $\mathrm{SO_7}$ on $\mathbb{C}^{\oplus 3} \otimes \O$ is $\mathrm{SO_4}$. The generic isotropy group for the action of $\mathrm{SL}_{3}$ on $M_{3}(\C)$ by transpose-conjugation is $\mathrm{O}(3)$. We deduce that the generic isotropy group for the action of  $\mathrm{SO_7} \times \mathrm{SL}_{3}$ on $\J_{3}(\O)$ contains $\mathrm{SO}_4 \times \mathrm{O}(3)$, which has dimension $9$. As a consequence,  the triplet $(\J_{3}(\O), \mathrm{SO_7} \times \mathrm{SL}_{3} \times \C^*, \SS)$ is not prehomogeneous and the generic orbit of $\mathrm{SO_7} \times \mathrm{SL}_{3} \times \C^*$ in $\J_{3}(\O)$ has codimension at least $27- (21+8+1 - (3+6)) = 6$.
\end{rem}

\end{subsection}

\begin{subsection}{A tautological rank $4$ sheaf on the Ogievetski\^i-Dray-Manogue sextic}

The main result of this section is related to the cokernel of $\begin{pmatrix} \lambda_1.I_8 & L_c & {}^{t} L_b \\ {}^{t} L_c & \lambda_2.I_8 & L_a \\ L_b & {}^{t} L_a & \l3.I_8 \end{pmatrix}$. In the following, we work over the polynomial ring $R = \mathbb{C}[x_1, \ldots, x_{27}]$. We let $c = x_1e_1+ \ldots + x_8e_8$, $b = x_{9}e_1 + \ldots + x_{16}e_n$, $a =x_{17}e_1 + \ldots + x_{24}e_n$, $\lambda_1 = x_{25}e_1$, $\lambda_2 = x_{26}e_1$ and $\l3 = x_{27}e_1$, where $e_1, \ldots, e_8$ is the canonical basis of $\mathbb{O}$.

\begin{theo}
Let $\mathrm{M_{\O}} = \begin{pmatrix} \lambda_1.I_8 & L_c & {}^{t} L_b \\ {}^{t} L_c & \lambda_2.I_8 & L_a \\ L_b & {}^{t} L_a & \l3.I_8 \end{pmatrix}$ in $\mathcal{M}_{24,24}(R)$ and let $\mathrm{E_{ODM}}$ be the cokernel of:
$$ \mathbb{C}^{24} \otimes \OO_{\mathbb{P}^{26}}(-1) \stackrel{\mathrm{M_{\O}}}\barrow \mathbb{C}^{24} \otimes \OO_{\mathbb{P}^{26}}.$$
The sheaf $\mathrm{E_{ODM}}$ is a $\mathrm{SO}_{7} \times\mathrm{SL}_{3}$-equivariant sheaf supported on $\SS$ and has generically rank $4$ on $\SS$.

\end{theo}

\begin{proof}
Theorem \ref{prop5} insures that the degeneracy of $\mathrm{M_{\O}}$ is $\SS$, so that $\mathrm{E_{ODM}}$ is supported on $\SS$. Note that the determinant of $\mathrm{M_{\O}}$ is of degree $24 = 4 \times \deg(\SS)$. In particular, we have $\mathrm{rank}(\mathrm{E_{ODM}}) = 4$. Let us now prove that $\mathrm{M_{\O}}$ is $\mathrm{SO}_{7} \times\mathrm{SL}_{3}$-equivariant. The $\mathrm{SL}_3$-equivariance is obvious given the definition of the $\mathrm{SL}_3$ action. We are left to prove the $\mathrm{SO}_7$-equivariance. Let $T_1 \in \mathrm{SO}_7$, the triality principle insures that there exists $T_2 \in \mathrm{SO}_8$ such that:
\[ \forall \ (u,v) \in \O^2, \ T_2(uv) = T_1(u)T_2(v).\]
It easily follows that for any $z \in \O$, we have:
\[ L_{T_1(z)} = T_2 L_z T_2^{-1}.\]
As a consequence, for any $A \in \J_{3}(\O)$ and any $T_1 \in \mathrm{SO}_7$ as above, we have:
\[ T_1.\mathrm{M_{\O}} = \begin{pmatrix} T_2 & & \\ & T_2 & \\ & & T_2 \end{pmatrix} \mathrm{M_{\O}} \begin{pmatrix} T_2^{-1} & & \\ & T_2^{-1} & \\ & & T_2^{-1} \end{pmatrix}.\]
This proves that $\mathrm{M_{\O}}$ is also $\mathrm{SO}_7$-equivariant. We can conclude that its cokernel, $\mathrm{E_{ODM}}$ is $\mathrm{SL}_3 \times \mathrm{SO}_7$-equivariant.
\end{proof}

\begin{rem}
The complete description of the higher degeneracy loci of $\mathrm{M_{\O}}$ seems to be an interesting question. It seems however to be also a difficult one, as the action of $\mathrm{SO}_{7} \times\mathrm{SL}_{3}$ on $\mathbb{P}^{26}$ has infinitely many orbits (see remark \ref{rem1}).
\end{rem}

\end{subsection}
\end{section}

\begin{section}{The geometry of the twisted eigenvalue problem for $\J_{3}(\O)$}

The motivation for introducing the twisted eigenvalue problem is the lack of <<prehomogeneity>> of the space $\big(\mathbb{P}(\J_{3}(\O)), \mathrm{SO}_7 \times \mathrm{SL}_3 \times \C^*, \SS \big)$. This lack of prehomogeneity is partly explained by the relatively big codimension of the subgroup of $\mathrm{E}_{6}$ generated by $\mathrm{SO}_{7}$ and $\mathrm{SL}_3$. On the other hand, it is well-known (see chapter $14$ of \cite{cali}) that $\mathrm{E}_6$ is generated by $\mathrm{Spin}_8$ and $\mathrm{SL}_{3}$. Thus, one could hope for an eigenvalue problem having a larger group of symmetries. As it happens, the symmetry group of the twisted octonionic eigenvalue problem contains the subgroup of $\mathrm{SL}_{27}$ generated by $\mathrm{Spin}_{7}$ and $\mathrm{SL}_3$ \footnote{The action of $\mathrm{Spin}_7$ on $\J_{3}(\O)$ that we exhibit is not the standard one and does not factor through an embedding of $\mathrm{Spin}_7$ in $\mathrm{E}_6$. As a matter of fact, we haven't been able to identify the subgroup of $\mathrm{SL}_{27}$ generated by $\mathrm{Spin}_{7}$ and $\mathrm{SL}_3$.} and we prove that $\J_{3}(\O)$ is prehomogeneous for this action. We also show that this twisted problem is singular on the union of the Cartan cubic and a sextic hypersurface. To the best of our knowledge, this geometry has never appeared before in the literature. 
\smallskip

\begin{subsection}{The twisted octonionic eigenvalue problem}
Let $A = \begin{pmatrix} \lambda_1 & c & \overline{b} \\ \overline{c} & \lambda_2 & a \\ b & \overline{a} & \l3 \end{pmatrix} \in \J_{3}(\O)$. We say that $\begin{pmatrix} x \\ y \\ z \end{pmatrix} \in \O^3$ is in the \textit{<<twisted kernel>>} of $A$ if:
\begin{equation*}
\left\{
\begin{split}
\lambda_1 x + y c + \overline{b} z & \ = 0 \\
x \overline{c} + \lambda_2 y + a z & \ = 0 \\
b x + \overline{a} y + \lambda_3 z & \ = 0\\
\end{split}
\right.
\end{equation*}
We say that $A \in \J_{3}(\O)$ is twisted-singular if the twisted kernel of $A$ is not zero. We notice that $\begin{pmatrix} \lambda_1 & c & \overline{b} \\ \overline{c} & \lambda_2 & a \\ b & \overline{a} & \l3 \end{pmatrix}$ is twisted-singular if and only if the $24*24$ matrix $N_{A} = \begin{pmatrix} \lambda_1.I_8 & R_c & {}^{t} L_b \\ {}^{t} R_c & \lambda_2.I_8 & L_a \\ L_b & {}^{t} L_a & \l3.I_8 \end{pmatrix}$ is singular, where $R_c$ is the $8*8$ matrix representing the right multiplication by $c$ in $\O$. Our first result related to the twisted eigenvalue problem is the:

\begin{theo} \label{mopi2}
The degeneracy locus of the matrix $N_A$ is given by the equation:

\[ \bigg(\det_{\O}(A)-2\mathrm{Re}(cab) + 2 \mathrm{Re}(\overline{c}ba) \bigg)^4 \times \mathfrak{S}^2 = 0,\]
where:
\begin{equation*}
\begin{split}
\mathfrak{S} & = \left(\ll_1|a|^2 + \ll_2|b|^2 + \ll_3|c|^2 - \ll_1\ll_2\ll_3 \right)^2 \\
                              & \hspace{0.5cm} - 4\left(\ll_1|a|^2 + \ll_2|b|^2 + \ll_3|c|^2 - \ll_1\ll_2\ll_3 \right)\mathrm{Re}(c)\mathrm{Re}(ab) \\
                              & \hspace{0.5cm} + 4\left(|b|^2|a|^2\mathrm{Re}(c)^2 + |c|^2 \mathrm{Re}(ab)^2 - |a|^2|b|^2|c|^2 \right),
\end{split}
\end{equation*}
with $|z|^2 = z\overline{z}$  and $\mathrm{Re}(z) = \frac{z + \overline{z}}{2}$.
\end{theo}
Note that $\det_{\O}(A)-2\mathrm{Re}(cab) + 2 \mathrm{Re}(\overline{c}ba) = \det_{\O}(A')$, where $A' = \begin{pmatrix} \lambda_1 & \overline{c} & \overline{a} \\ c & \lambda_2 & b \\ a & \overline{b} & \l3 \end{pmatrix}$. We refrain from adopting such a notation for the cubic term appearing in the equation of the degeneracy locus of $N_A$ as it may be a source of confusions. 
We prove Theorem \ref{mopi2} using a strategy similar to that used for Theorem \ref{prop5}.

\begin{proof}
For $a,b,c \in \O$ and $\ll_1, \ll_2, \ll_3 \in \mathbb{C}$ such that $|a|,|b|,|c|, \ll_1, \ll_2, \ll_3 \neq 0$ and $|c|^2\ll_1 - \dfrac{|c|^4}{\ll_2} \neq 0$, we define the following $24*24$ matrices in block form:
\[ S_1 = \begin{pmatrix} \frac{1}{|c|^2} R_c & 0_{8,8} & 0_{8,8} \\ 0_{8,8}& I_{8}& 0_{8,8} \\ 0_{8,8} & 0_{8,8}& \frac{1}{|a|^2} {}^{t} L_a \end{pmatrix} \ \ \textrm{and} \ \   S_2 = \begin{pmatrix} I_8 & \frac{|c|^2}{\ll_2} & 0_{8,8} \\ 0_{8,8}& I_{8}& 0_{8,8} \\ \frac{1}{\delta} B & \frac{|a|^2}{\ll_2}I_8 & I_8 \end{pmatrix},\]
where $B = L_aL_b R_c - \dfrac{|a|^2|c|^2}{\ll_2}.I_8$, $\delta = |c|^2\ll_1 - \dfrac{|c|^4}{\ll_2}$, $0_{8,8}$ is the zero $8*8$ matrix and $I_8$ the identity matrix of size $8$.
One easily checks that:

\[ \begin{pmatrix} \lambda_1.I_8 & R_c & {}^{t} L_b \\ {}^{t} R_c & \lambda_2.I_8 & L_a \\ L_b & {}^{t} L_a & \l3.I_8 \end{pmatrix} = S_1S_2 \begin{pmatrix} \delta.I_8 & 0_{8,8} & 0_{8,8} \\ 0_{8,8} & \ll_2.I_8 & 0_{8,8} \\ 0_{8,8} & 0_{8,8} & \nu.I_8 - \dfrac{1}{\delta} B{}^{t}B \end{pmatrix} {}^{t} S_2 {}^{t} S_1, \]
where $\nu = \ll_3|a|^2 - \frac{|a|^4}{\ll_2}$. Since $\det(S_1) =\dfrac{1}{|a|^8|c|^8}$ and $ \det(S_2) = 1$, we get:
\begin{equation*}
\begin{split}
\det \begin{pmatrix} \lambda_1.I_8 & R_c & {}^{t} L_b \\ {}^{t} R_c & \lambda_2.I_8 & L_a \\ L_b & {}^{t} L_a & \l3.I_8 \end{pmatrix} & = \dfrac{1}{|a|^{16} |c|^{16}} \det \left( \delta \nu \ll_2.I_8 - \ll_2.B{}^{t} B \right)\\
                                                                    & = \det \left((\ll_1 \ll_2 \ll_3 - \ll_1|a|^2 - \ll_2 |b|^2 - \ll_3 |c|^2).I_8 + (L_aL_bR_c +{}^{t}R_c {}^tL_b {}^{t} L_a) \right)\\
                                                                    & = \det(\det_{\mathbb{O}}(A) - 2 \mathrm{Re}(cab)).I_8 + (L_aL_bR_c + {}^{t}R_c {}^tL_b {}^{t} L_a)).
\end{split}
\end{equation*}

The expression $\det(\det_{\mathbb{O}}(A) - 2 \mathrm{Re}(cab)).I_8 + (L_aL_bR_c + {}^{t}R_c {}^tL_b {}^{t} L_a))$ has no denominators so that the above equality holds for all  $a,b,c \in \O$ and $\ll_1, \ll_2, \ll_3 \in \mathbb{C}$. We deduce the following:

\smallskip

\noindent The determinant of $\begin{pmatrix} \lambda_1.I_8 & R_c & {}^{t} L_b \\  {}^{t} R_c & \lambda_2.I_8 & L_a \\ L_b & {}^{t} L_a & \ll_3.I_8 \end{pmatrix}$ vanishes if and only if $ 2 \mathrm{Re}(cab) - \det_{\O}(A)$ is a (complex) eigenvalue of $(L_aL_bR_c +{}^{t}R_c {}^tL_b {}^{t} L_a)$.

\bigskip

We shall now determine the (complex) eigenvalue of $(L_aL_bR_c +{}^{t} R_c {}^tL_b {}^{t} L_a)$. This means we are looking for the $\k \in \C$ such that:
\begin{equation} \label{Mopi}
\exists x \in \O, \ x \neq 0 \ \ a.(b.(x.c)) + (\overline{b}.(\overline{a}.x)).\overline{c} = \kappa.x
\end{equation}
Let us consider a four dimensional quaternionic subalgebra of $\O$ which contains $a$ and $c$. Let $e \in \O$ be a vector orthogonal to $M$ such that $\O = M \oplus M.e$. We write $b = b_0 + b_1.e$ and $x = x_0 + x_1.e$, with $b_0,b_1,x_0,x_1 \in M$. Using the relations:
$$ u.(v.e) = (v.u).e, \ (u.e).(v.e) = -\overline{v}.u, \ u.e = e. \overline{u} \ \textrm{and} \ (u.e).v = (u. \overline{v}).e,$$
valid for all $(u,v) \in M^2$, we find that equation \ref{Mopi} is equivalent to the system:
\begin{equation*}
\left\{
 \begin{split} a.b_0. x_0.c + \overline{b_0}.\overline{a}.x_0.\overline{c} + a.\overline{x_1}.b_1.\overline{c} - {a}.c.\overline{x_1}.b_1 &= \kappa.x_0\\
x_1.\overline{c}.b_0.a + b_1.\overline{c}.\overline{x_0}.a + x_1.\overline{a}.\overline{b_0}.c - b_1.\overline{x_0}.a.c & = \kappa x_1.\\
\end{split}\right.
\end{equation*}
which, by conjugating the second line, is equivalent to:
\begin{equation*} 
\left\{
\begin{split}
a.b_0. x_0.c + \overline{b_0}.\overline{a}.x_0.\overline{c} + a.\overline{x_1}.b_1.\overline{c} - {a}.c.\overline{x_1}.b_1 &= \kappa.x_0\\
\overline{a}.\overline{b_0}.c.\overline{x_1} + \overline{c}.b_0.a.\overline{x_1}+ \overline{a}.x_0.c.\overline{b_1}- \overline{c}. \overline{a}.x_0 \overline{b_1} & = \kappa.\overline{x_1}
\end{split}
\right.
\end{equation*}
which, in turn, is equivalent to:
\begin{equation} \label{sysmopi}
\left\{
\begin{split}
a.b_0. x_0.c + \overline{b_0}.\overline{a}.x_0.\overline{c} + a.\overline{x_1}.b_1.\overline{c} - {a}.c.\overline{x_1}.b_1 &= \kappa.x_0\\
2 \mathrm{Re}(\overline{c} b a) \overline{x_1}+ \overline{a}.x_0.c.\overline{b_1}- \overline{c}. \overline{a}.x_0 \overline{b_1} & = \kappa.\overline{x_1}
\end{split}
\right.
\end{equation}
as $\mathrm{Re}(\overline{c}b_0a) = \mathrm{Re}(\overline{c}ba).$ It happens that the matrix form of this system can not be exploited as easily as in the proof of Theorem \ref{prop5}. We will nevertheless prove that for all $a,b,c \in \O$, the complex $2.\mathrm{Re}(\overline{c}ba)$ is an eigenvalue of multiplicity at least $4$ of $(L_aL_bR_c + {}^{t}R_c {}^tL_b {}^{t} L_a)$. Let $V_1$ be the space of vectors $\begin{pmatrix} x_0 \\ \overline{x_1} \end{pmatrix}$ such that:
\[ x_0 = 0 \ \textrm{et} \ \overline{x_1}.b_1.\overline{c} =  c.\overline{x_1}.b_1\]
One directly checks with the help of system (\ref{sysmopi}) that for any such vectors:
\[ (L_aL_bR_c + {}^{t}R_c {}^tL_b {}^{t} L_a)  \begin{pmatrix} x_0 \\ \overline{x_1} \end{pmatrix}  = 2 \mathrm{Re}(\overline{c}ba) \begin{pmatrix} x_0 \\ \overline{x_1} \end{pmatrix}.\] 
The quaternionic structure of $M$ insures that $\dim V_1 \geq 2$.  Let $V_2$ be the space of vectors $\begin{pmatrix} x_0 \\ \overline{x_1} \end{pmatrix}$ such that:
\[ \overline{a}.x_0.c = \overline{c}.\overline{a}.x_0 \ \textrm{and} \ {a}.{b_0}.x_0.{c} = {a}.c.\overline{x_1}.b_1\]
One directly checks with the help of system (\ref{sysmopi}) that for any such vectors:
\[ (L_aL_bR_c + {}^{t}R_c {}^tL_b {}^{t} L_a)  \begin{pmatrix} x_0 \\ \overline{x_1} \end{pmatrix}  = 2 \mathrm{Re}(\overline{c}ba) \begin{pmatrix} x_0 \\ \overline{x_1} \end{pmatrix}.\] 
The quaternionic structure of $M$ insures that $\dim V_2 \geq 2$. We note that $V_1 \cap V_2 = \begin{pmatrix} 0 \\ 0 \end{pmatrix}$ which implies $\dim V_1 \oplus V_2 \geq 4$. As a consequence, $2 \mathrm{Re}(\overline{c}ba)$ is an eigenvalue of multiplicity at least $4$ of $(L_aL_bR_c + {}^{t}R_c {}^tL_b {}^{t} L_a)$. In particular, we may write:
\[ \det \bigg( \kappa.I_8 - L_aL_bR_c + {}^{t}R_c {}^tL_b {}^{t} L_a \bigg) = (\kappa - 2.\mathrm{Re}(\overline{c}ba))^4.T(a,b,c,\kappa), \]
where $T(a,b,c,\kappa)$ is weighted-homogeneous of degree $12$ with weights $(1,1,1,3)$.

\bigskip

We can now conclude the proof of the Theorem as we remember that the determinant of $\begin{pmatrix} \lambda_1.I_8 & R_c & {}^{t} L_b \\ {}^{t} R_c & \lambda_2.I_8 & L_a \\ L_b & {}^{t} L_a & \ll_3.I_8 \end{pmatrix}$ vanishes if and only if $2 \mathrm{Re}(cab) - \det_{\O}(A)$ is a (complex) eigenvalue of $(L_aL_bR_c +{}^{t}R_c {}^tL_b {}^{t} L_a)$. Hence the degeneracy locus of $N_{A}$ is scheme-theoretically given by the equation:

$$ \bigg(\det_{\O}(A)-2 \mathrm{Re}(cab) + 2 \mathrm{Re}(\overline{c}ba)\bigg)^4 \times \mathfrak{T}= 0,$$
where $\mathfrak{T}$ is the degree $12$ homogeneous polynomial in $a,b,c$ defined by $\mathfrak{T} = T(a,b,c,2 \mathrm{Re}(cab) - \det_{\O}(A))$. One easily proves that $\mathfrak{T} = \mathfrak{S}^2$ is a square and the explicit equation:
\begin{equation*}
\begin{split}
\mathfrak{S} & = \left(\ll_1|a|^2 + \ll_2|b|^2 + \ll_3|c|^2 - \ll_1\ll_2\ll_3 \right)^2 \\
                              & \hspace{0.5cm} - 4\left(\ll_1|a|^2 + \ll_2|b|^2 + \ll_3|c|^2 - \ll_1\ll_2\ll_3 \right)\mathrm{Re}(c)\mathrm{Re}(ab) \\
                              & \hspace{0.5cm} + 4\left(|b|^2|a|^2\mathrm{Re}(c)^2 + |c|^2 \mathrm{Re}(ab)^2 - |a|^2|b|^2|c|^2 \right),
\end{split}
\end{equation*}
has been first found by Jonathan Hauenstein using the software \textit{Bertini}.

\end{proof}

\end{subsection}

\begin{subsection}{A quasi-tautological rank $4$ sheaf on the Cartan cubic}

We shall first explain how $\mathrm{Spin}_7$ and $\mathrm{SL}_3$ act on $N_{A} = \begin{pmatrix} \lambda_1.I_8 & R_c & {}^{t} L_b \\ {}^{t}R_c & \lambda_2.I_8 & L_a \\ L_b & {}^{t} L_a & \l3.I_8 \end{pmatrix}.$ The action of $\mathrm{SL}_3$ is again by transpose-conjugation and is the same as in section 2.2. To describe the action of $\mathrm{Spin}_7$, we shall need a few more facts on triality (we refer to \cite{triality} for more details on the triality principle). We recall that there is an embedding (the triality embedding) of $\mathrm{Spin}_8$ in $\mathrm{SO}_8 \times \mathrm{SO}_8 \times \mathrm{SO}_8$ defined as follows:
\[ \mathrm{Spin}_8 = \{(T_1,T_2,T_3) \in \big(\mathrm{SO}_8\big)^3, \ T_1(x)T_2(y) = T_3(xy), \ \textrm{for all} \ (x,y) \in \O^2 \}.\]
There are three copies of $\mathrm{Spin}_7$ inside $\mathrm{Spin}_8$ which have remarkable group-theoretic properties with respect to the triality embedding. We are interested in one of them:
\[ \mathrm{Spin}_7 = \{(T_1,T_2,T_1) \in \big(\mathrm{SO}_8\big)^3, \ T_1(x)T_2(y) = T_1(xy), \ \textrm{for all} \ (x,y) \in \O^2 \}.\]
The projection map:
\begin{equation*}
\begin{split}
& \mathrm{Spin}_7 \barrow \mathrm{SO}_{8} \\
& (T_1,T_2,T_1) \longrightarrow T_2
\end{split}
\end{equation*}
is a double cover of $\mathrm{SO}_7$ embedded in $\mathrm{SO_8}$ as the stabilizer of $1 \in \mathbb{O}$. For any $T \in \mathrm{SO}_{8}$, we denote by $\K_{T}$ the element of $\mathrm{SO}_8$ defined by $\K_{T}(x) = \overline{T(\overline{x})}$. It is easily checked that if $(T_1,T_2,T_1)$ is a triality, then $(\K_{T_1},T_1,T_2)$ and $(T_2,\K_{T_1}, \K_{T_1})$ are equally trialities (but, in general, they lie in different copies $\mathrm{Spin}_7 \subset \mathrm{Spin}_8$). 

\bigskip

We now come back to the twisted eigenvalue problem. Let $A = \begin{pmatrix} \lambda_1 & c & \overline{b} \\ \overline{c} & \lambda_2 & a \\ b & \overline{a} & \l3 \end{pmatrix} \in \J_{3}(\O)$ and let $\begin{pmatrix} x \\ y \\ z \end{pmatrix} \in \O^3$ is in the twisted kernel of $A$, namely:
\begin{equation*}
\left\{
\begin{split}
\lambda_1 x + y c + \overline{b} z & \ = 0 \\
x \overline{c} + \lambda_2 y + a z & \ = 0 \\
b x + \overline{a} y + \lambda_3 z & \ = 0\\
\end{split}
\right.
\end{equation*}
A quick computation shows that for any $(T_1,T_2,T_1) \in \mathrm{Spin}_7$, the vector $\begin{pmatrix} T_1(x) \\ T_1(y) \\ T_2(z) \end{pmatrix} \in \O^3$ is in the twisted kernel of $\begin{pmatrix} \lambda_1 & T_2(c) & \overline{\K_{T_1}(b)} \\ \overline{T_2(c)} & \lambda_2 & T_1(a) \\ \K_{T_1}(b) & \overline{T_1(a)} & \l3 \end{pmatrix}$, that is:
\begin{equation*}
\left\{
\begin{split}
\lambda_1 T_1(x) + T_1(y) T_2(c) + T_1(\overline{b}) T_2(z) & \ = 0 \\
T_1(x) \overline{T_2(c)} + \lambda_2 T_1(y) + T_1(a) T_2(z) & \ = 0 \\
\K_{T_1}(b) T_1(x) + \overline{T_1(a)} y + \lambda_3 T_2(z) & \ = 0\\
\end{split}
\right.
\end{equation*}
For any $(T_1,T_2,T_1) \in \mathrm{Spin_7}$ and any $A = \begin{pmatrix} \lambda_1 & c & \overline{b} \\ \overline{c} & \lambda_2 & a \\ b & \overline{a} & \l3 \end{pmatrix} \in \J_{3}(\O)$,  we then define:
\[ (T_1,T_2,T_1).A = \begin{pmatrix} \lambda_1 & T_2(c) & \overline{\K_{T_1}(b)} \\ \overline{T_2(c)} & \lambda_2 & T_1(a) \\ \K_{T_1}(b) & \overline{T_1(a)} & \l3 \end{pmatrix}.\]
This defines an action of $\mathrm{Spin}_7$ on $\J_{3}(\O)$ which preserves the twisted eigenvalue problem. The main result of this section is related to the cokernel of $ \begin{pmatrix} \lambda_1.I_8 & R_c & {}^{t} L_b \\  {}^{t} R_c & \lambda_2.I_8 & L_a \\ L_b & {}^{t} L_a & \l3.I_8 \end{pmatrix}$. As in section 2.3, we work over the polynomial ring $R = \mathbb{C}[x_1, \ldots, x_{27}]$. We let $c = x_1e_1+ \ldots + x_8e_8$, $b = x_{9}e_1 + \ldots + x_{16}e_n$, $a =x_{17}e_1 + \ldots + x_{24}e_n$, $\lambda_1 = x_{25}e_1$, $\lambda_2 = x_{26}e_1$ and $\l3 = x_{27}e_1$, where $e_1, \ldots, e_8$ is the canonical basis of $\mathbb{O}$.

\begin{theo} \label{mopi4}
Let $N_{\O} = \begin{pmatrix} \lambda_1.I_8 & R_c & {}^{t} L_b \\  {}^{t}R_c & \lambda_2.I_8 & L_a \\ L_b & {}^{t} L_a & \l3.I_8 \end{pmatrix}$ in $\mathcal{M}_{24,24}(R)$ and let $E_{\O}$ be the cokernel of:
$$ \mathbb{C}^{24} \otimes \OO_{\mathbb{P}^{26}}(-1) \stackrel{N_{\O}}\barrow \mathbb{C}^{24} \otimes \OO_{\mathbb{P}^{26}}.$$
\begin{enumerate}
\item The sheaf $E_{\O}$ is $\langle \mathrm{Spin}_{7}, \mathrm{SL}_{3} \rangle$-equivariant (where $\langle \mathrm{Spin_7}, \mathrm{SL}_3 \rangle$ is the subgroup of $\mathrm{SL}_{24}$ generated by $\mathrm{SL}_3$ and $\mathrm{Spin_7}$ as above).
\item It is supported on the union of a cubic hypersurface isomorphic to $\CC_{\O}$ and the degree $6$ hypersurface given by the equation $\mathfrak{S} = 0$. Both hypersurfaces are $\langle \mathrm{Spin}_{7}, \mathrm{SL}_{3} \rangle$-invariant.
\item The restriction of $E_{\O}$ to the cubic hypersurface has generically rank $4$ while its restriction to the sextic $\mathfrak{S} = 0$ has generically rank $2$.
\end{enumerate}
\end{theo}
In the following, we shall also denote by $\mathfrak{S}$ the hypersurface defined by $\mathfrak{S} = 0$.

\begin{proof}
Theorem \ref{mopi2} insures that the degeneracy locus of $N_{\O}$ is scheme-theoretically given by $(\mathrm{Det}_{\O} - \mathrm{Re}(cab) + \mathrm{Re}(\overline{c}ba) )^4 \times \mathfrak{S}^2 = 0$. We already noted that $(\det_{\mathbb{O}}(A) - \mathrm{Re}(cab) + \mathrm{Re}(\overline{c}ba) ) = \det_{\mathbb{O}} A'$, where $A' = \begin{pmatrix} \lambda_1 & \overline{c} & \overline{a} \\ c & \lambda_2 & b \\ a & \overline{b} & \l3 \end{pmatrix}$. As a consequence, the equation defines a cubic isomorphic to the Cartan cubic, which we denote by $\mathcal{C}'_{\mathbb{O}}$. The sheaf $E_{\O}$ is then supported on the union of $\CC'_{\O}$ and the degree $6$ hypersurface $\mathfrak{S}$. The rank of the restriction of $E_{\O}$ to $\CC'_{\O}$ is exactly $4$ as $(\mathrm{Det}_{\O} - \mathrm{Re}(cab) + \mathrm{Re}(\overline{c}ba))$ does not divide $\mathfrak{S}$. The same argument shows that the rank of $E_{\O}$ restricted to $\mathfrak{S}$ is $2$.

\bigskip

The $\mathrm{SL}_3$-equivariance of $N_{\O}$ follows from the definition of the $\mathrm{SL}_3$ action (by transpose-conjugation). Let us prove that $N_{\O}$ is also $\mathrm{Spin}_7$-equivariant. For any $(T_1,T_2,T_1) \in \mathrm{Spin}_7$, we know that $(\K_{T_1}, T_1, T_2)$ is also in $\mathrm{Spin}_8$. Hence, for any $(x,y,z,w) \in \O^4$ and any $(T_1,T_2,T_1) \in \mathrm{Spin}_7$, we have:

\[ R_{\overline{T_2(x)}} = T_1 {}^{t} R_{x} T_1^{-1}, \ L_{\overline{\K_{T_1}(y)}} = T_1 {}^{t} L_y T_2^{-1}, \ L_{T_1(z)} = T_1 L_z T_2^{-1} \ \textrm{and} \ L_{\overline{T_1}(w)} = T_2 {}^{t} L_{w} T_1^{-1} .\]
As a consequence, for any $(T_1,T_2,T_1) \in \mathrm{Spin}_7$, we find:
\[ (T_1,T_2,T_1). N_{\O} =  \begin{pmatrix} T_1 & & \\ & T_1 & \\ & & T_2 \end{pmatrix} N_{\O} \begin{pmatrix} T_1^{-1} & & \\ & T_1^{-1} & \\ & & T_2^{-1} \end{pmatrix},\]
which proves the $\mathrm{Spin}_7$-equivariance of $N_{\O}$. In fine, we can conclude that $E_{\O}$, the cokernel of $N_{\O}$, is $\langle \mathrm{Spin_7}, \mathrm{SL}_3 \rangle$-equivariant, where $\langle \mathrm{Spin_7}, \mathrm{SL}_3 \rangle$ is the subgroup of $\mathrm{SL}_{24}$ generated by $\mathrm{SL}_3$ and $\mathrm{Spin_7}$. Since $E_{\O}$ is $\langle \mathrm{Spin}_{7}, \mathrm{SL}_{3} \rangle$-equivariant, its support must be $\langle \mathrm{Spin}_{7}, \mathrm{SL}_{3} \rangle$-invariant. This proves that the sextic $\mathfrak{S}$ is $\langle \mathrm{Spin}_{7}, \mathrm{SL}_{3} \rangle$-invariant.

\end{proof}

\end{subsection}

\begin{rem}
On easily checks that the jumping locus of the restriction of $E_{\O}$ to the cubic $\mathcal{C}'_{\mathbb{O}}$ is not the singular locus of $\mathcal{C}'_{\mathbb{O}}$ (that is a copy of $\OP^2$) : it has strictly smaller codimension. it would be interesting to know if we can give an interpretation of $E_{\O}|_{A}$ for any $A$ in the smooth locus of $\mathcal{C}'_{\mathbb{O}}$ using only the octonionic geometry in $\O$ (as we can do for the tautological sheaves described in the introduction for the associative exceptional algebras $\R \otimes \C, \C \otimes \C$ and $\H \otimes \C$).
\end{rem}

\begin{subsection}{On the automorphism group of twisted eigenvalue problem}
It is well known (see the section 14 of \cite{cali}) that the subgroup of $\mathrm{E}_6$ generated by $\mathrm{SL}_3$ and $\mathrm{Spin_8}$ is $\mathrm{E}_6$ itself. In fact, one can prove that the subgroup of $\mathrm{E}_6$ generated by $\mathrm{SL}_3$ and $\mathrm{Spin_7}$ is also $\mathrm{E}_6$. The following argument has been communicated to me by Robert Bryant. Let $h \in \mathrm{SL}_3$ defined by:
\[ h = \begin{pmatrix} 0 & -1 & 0 \\ -1 & 0 & 0 \\ 0 & 0 & -1 \end{pmatrix}.\]
Let us denote by $K_2$ the copy of $\mathrm{Spin_7}$ embedded in $\mathrm{Spin}_8$  we considered in the previous section:
\[ K_2 = \{(T_1,T_2,T_1) \in \big(\mathrm{SO}_8\big)^3, \ T_1(x)T_2(y) = T_1(xy), \ \textrm{for all} \ (x,y) \in \O^2 \}.\]
It is easily checked that we have $hK_2h^{-1} = K_1$ as subgroups of $\mathrm{E}_6$, where $K_1$ is the copy of $\mathrm{Spin_7}$ embedded in $\mathrm{Spin}_8$ as:
\[ K_1 = \{(T_2,T_1,T_1) \in \big(\mathrm{SO}_8\big)^3, \ T_2(x)T_1(y) = T_1(xy), \ \textrm{for all} \ (x,y) \in \O^2 \}.\] It is known that these two copies of $\mathrm{Spin}_7$ are sufficient to generate $\mathrm{Spin}_8$ (see the section $14$ of \cite{cali}) and we are done.

\begin{prop}
The action of $\mathrm{Spin}_7$ on $\J_{3}(\O)$ we considered in Theorem \ref{mopi4} does not factor through the action $\mathrm{E}_6$ on $\J_{3}(\O)$.
\end{prop}
\begin{proof}
Assume by absurd that it does. Then the remark above shows that $\langle \mathrm{Spin}_{7}, \mathrm{SL}_{3} \rangle = \mathrm{E}_6$. One immediately checks that the action of $\langle \mathrm{Spin}_{7}, \mathrm{SL}_{3} \rangle$ we defined on $\J_{3}(\O)$ preserves the rank, when we see elements of $\J_{3}(\O)$ as $24*24$ matrices. Since any element of $\J_{3}(\O)$ can be diagonalized by the action of $\mathrm{E}_6$ (see the chapter 14 of \cite{cali}), this would prove that any element of $\J_{3}(\O)$ has rank $0,8,16$ or $24$ as a $24*24$ matrix. But we know that a generic point on $\mathfrak{S} = 0$ has rank $22$ (see Theorem \ref{mopi4}), a contradiction.

\end{proof}

We now state the main result of this section:

\begin{theo} \label{homogeneity}
The action of $\mathrm{Aut}^{0}(\mathfrak{S}) \times \C^*$ on $\J_{3}(\O)$ is prehomogeneous.
\end{theo}
This result is slightly surprising as the subgroup of $\mathrm{SL}_{27}$ generated by $\mathrm{Spin}_7(\C)$ and $\mathrm{SL}_3(\C)$ is not $\mathrm{E}_6$ (see the above proposition). One also notices that the prehomogeneous space $(\J_{3}(\O), \mathrm{Aut}^{0}(\mathfrak{S})) \times \C^*$ does not appear in the list of Kimura-Sato \cite{kimura}. This is not surprising as it is unlikely to be an irreducible prehomogeneous space. It thus provides an interesting example of a non-irreducible prehomogeneous vector space which does not split as a sum of lower dimensional prehomoegeneous vector spaces. Theorem \ref{homogeneity} also shows that the geometry of the twisted octonionic eigenvalue problem differes eminently from that of the untwisted eigenvalue problem : the former is much more <<symmetric>> than the latter. Indeed, we noticed in section 2.2  that $ \mathrm{Aut}^{0}(\SS) \times \C^* = \mathrm{SO}_7 \times \mathrm{SL}_3 \times \C^*$ does not act prehomogeneously on $\J_{3}(\O)$.

\bigskip

\begin{proof}
We have shown in the previous section that the subgroup $G$ of $\mathrm{SL}_{27}$ generated by $\mathrm{SL}_3$ and $\mathrm{Spin}_7$ is included in $\mathrm{Aut}^{0}(\mathfrak{S})$. We are going to prove that the action of $\mathbb{C}^* \times G$ on $\J_{3}(\O)$ is prehomogeneous. We first start with:

\begin{lem} \label{adams}
Let $\{1,i,j,k,l,m,n,o\}$ be a basis of $\O$ over $\C$. We fix a splitting $\O = \H \oplus \H.l$, with $\H = \mathrm{Vect}(1,i,j,k)$. The group $\mathrm{Spin}_7$ acts transitively on the following varieties:
\begin{equation*}
\begin{split}
& S^6 \times \{(x +y.l) \in  \H \oplus \H.l, \ \langle x,y \rangle =0, x \neq 0, y \neq 0 \} \subset \mathrm{Im}(\O) \times \O, \\
& S^6 \times \left( \H \backslash \{0\} \oplus \{0 \} \right) \subset \mathrm{Im}(\O) \times \O , \\
& S^6 \times \left(\{0 \} \oplus \H.l \backslash \{0\} \right) \subset \mathrm{Im}(\O) \times \O,
\end{split}
\end{equation*}
where $S^6$ denotes the subvariety of $\mathrm{Im}(\O)$ defined by $|z|^2 = 1$.
\end{lem}
This lemma is well-known, we refer to corollary 5.4 in \cite{JFadams}. We recall a proof for the convenience of the reader since the results of \cite{JFadams} are stated over $\mathbb{R}$ and there are some minor modifications to be made over $\mathbb{C}$.

\begin{proof}
The projection map:
\begin{equation*}
\begin{split}
& \mathrm{Spin}_7 \barrow \mathrm{SO}_{8} \\
& (T_1,T_2,T_1) \longrightarrow T_2
\end{split}
\end{equation*}
is a double cover of $\mathrm{SO}_7$ embedded in $\mathrm{SO_8}$ as the stabilizer of $1 \in \mathbb{O}$. Hence, $\mathrm{Spin}_7$ acts transitively on $S^6$ and the stabilizer of any point in $S^6$ is $\mathrm{Spin}_6$. We know that $\mathrm{Spin}_6 \simeq \mathrm{SL}_4$ when we make the identification $\H \oplus \H.l = \mathbb{C}^4 \oplus (\mathbb{C}^4)^*$. Furthermore, for any $(x+y.l) \in  \H \oplus \H.l$, the scalar product $\langle x,y \rangle$ is the perfect pairing between $\mathbb{C}^4$ and $(\mathbb{C}^4)^*$. This pairing is obviously invariant for the product of the natural and the contragedient representation of $\mathrm{SL}_4$. The action of $\mathrm{SL}_4$ on $\mathbb{C}^4 \backslash \{0\}$ is transitive and the stabilizer in $\mathrm{SL}_4$ of any non-zero $x \in \mathbb{C}^4$ contains $\mathrm{SL}_3$, which acts on $x^{\perp} \subset (\mathbb{C}^4)^*$. The transitive action of $\mathrm{SL}_3$ on $\mathbb{C}^3 \backslash \{0\}$ allows us to conclude the proof of the lemma.

\end{proof}

We can now proceed to the proof of Theorem \ref{homogeneity}. Recall that the action of $\mathrm{Spin}_7$ on $\J_{3}(\O)$ is given as follows. For any $(T_1,T_2,T_1) \in \mathrm{Spin_7}$ and any $A = \begin{pmatrix} \lambda_1 & c & \overline{b} \\ \overline{c} & \lambda_2 & a \\ b & \overline{a} & \l3 \end{pmatrix} \in \J_{3}(\O)$,  we have:
\[ (T_1,T_2,T_1).A = \begin{pmatrix} \lambda_1 & T_2(c) & \overline{\K_{T_1}(b)} \\ \overline{T_2(c)} & \lambda_2 & T_1(a) \\ \K_{T_1}(b) & \overline{T_1(a)} & \l3 \end{pmatrix}.\]

\noindent We start with $A = \begin{pmatrix} \lambda_1 & c & \overline{b} \\ \overline{c} & \lambda_2 & a \\ b & \overline{a} & \l3 \end{pmatrix}$ generic in $\J_{3}(\O)$. By lemma \ref{adams}, we can find $(T_1,T_2,T_1) \in \mathrm{Spin_7}$ such that $T_2(c) = r_1 + r_2.i$, with $r_1,r_2  \in \mathbb{C}$ non zeros (by genericity of $A$). Hence, we have:
\[A_1 = (T_1,T_2,T_1).A = \begin{pmatrix} \lambda_1 & r_1 + r_2.i & \overline{b_1} \\ r_1- r_2.i & \lambda_2 & a' \\ b_1 & \overline{a'} & \l3 \end{pmatrix},\]
with $a' = T_1(a)$ and $b_1 = \mathrm{K}_{T_1}(b)$. Let us write $a' = a_0' + a_1'.l$ with $a_0',a_1' \in \H$. Since $A$ is assumed generic, we can arrange $T_1$ in the previous transformation so that $\langle r_1 - r_2.i, a_1' \rangle \neq 0$. We plan to make the transformation $a' \longleftarrow a' - \frac{ \langle a_0',a_1' \rangle}{\langle r_1 - r_2.i, a_1' \rangle}(r_1-r_2.i)$. We thus consider $H = \begin{pmatrix} 1 & 0 &  \frac{-\langle a_0',a_1' \rangle}{\langle r_1 - r_2.i, a_1' \rangle} \\ 0 & 1 & 0 \\ 0 & 0 & 1 \end{pmatrix} \in \mathrm{SL}_3$ and we have:

\[ A_2 = {}^{t} H A_1H = \begin{pmatrix} \lambda_1 & r_1 + r_2.i & \overline{b_2} \\ r_1- r_2.i & \lambda_2 & a'' \\ b_2 & \overline{a''} & {\l3}' \end{pmatrix},\]
with $a'' = a'-\frac{ \langle a_0',a_1' \rangle}{\langle r_1 - r_2.i, a_1' \rangle}(r_1-r_2.i)$, $b_2 \in \O$ and ${\l3}' \in \C$. If we put $a'' = a_0''+a_1'.l$ with $a_0'' \in \H$, then we have $\langle a_0'',a_1' \rangle = 0$ by construction.  By lemma \ref{adams}, we can find $(T_1,T_2,T_1) \in \mathrm{Spin_7}$ such that $T_2(r_1 + r_2.i) = r_1 + r_2.n$ and $T_1(a'') = r_3.i + r_4.n$, with $r_3,r_4  \in \mathbb{C}$ non zeros (by genericity of $A$). Hence, we have:
\[A_3 = (T_1,T_2,T_1).A_2 = \begin{pmatrix} \lambda_1 & r_1 + r_2.n & \overline{b_3} \\ r_1- r_2.n & \lambda_2 & r_3.i + r_4.n \\ b_3 & -r_3.i-r_4.n & {\l3}' \end{pmatrix},\]
 with $b_3 \in \O$. We now plan to make the transformation $r_3i.+r_4.n \longleftarrow r_3.i+ r_4.n + \frac{r_4}{r_2}(r_1-r_2.n)$. We thus consider $H = \begin{pmatrix} 1 & 0 &  \frac{r_4}{r_2} \\ 0 & 1 & 0 \\ 0 & 0 & 1 \end{pmatrix} \in \mathrm{SL}_3$ and we have:

\[ A_4 = {}^{t} H A_3H = \begin{pmatrix} \lambda_1 & r_1+ r_2.n& \overline{b_4} \\ r_1-r_2.n & \lambda_2 & r_5+r_3.i  \\ b_4 & r_5-r_.i & {\l3}'' \end{pmatrix},\]
with $r_5 = \frac{r_4}{r_2}r_1 \in \mathbb{C}$, ${\l3}'' \in \C$ and $b_4 \in \O$. By lemma \ref{adams}, we can find $(T_1,T_2,T_1) \in \mathrm{Spin_7}$ such that $T_2(r_1 + r_2n) = r_1 + r_2.i$ and $T_1(r_5+r_3.i) = r_5 + r_3.i$. Hence, we have:
\[A_5 = (T_1,T_2,T_1).A_4 = \begin{pmatrix} \lambda_1 & r_1 + r_2.i & \overline{b_5} \\ r_1- r_2.i & \lambda_2 & r_5 + r_3.i \\ b_5 & r_5-r_3.i& {\l3}'' \end{pmatrix},\]
 with $b_5 \in \O$. We now plan to make the transformation $r_1 + r_2.i \longleftarrow r_1+ r_2.i + \frac{r_2}{r_3}(r_5-r_3.i)$. We thus consider $H = \begin{pmatrix} 1 & 0 &  0 \\ 0 & 1 & 0 \\ \frac{r_2}{r_3} & 0 & 1 \end{pmatrix} \in \mathrm{SL}_3$ and we have:

\[ A_6 = {}^{t} H A_5H = \begin{pmatrix} {\lambda_1}' & r_6& \overline{b_6} \\ r_6 & \lambda_2 & r_5+r_3.i  \\ b_5 & r_5-r_3.i & {\l3}'' \end{pmatrix},\]
with $r_6 = r_1 +\frac{r_2}{r_3}r_5 \in \mathbb{C}$, ${\lambda_1}' \in \C$ and $b_6 \in \O$. By lemma \ref{adams}, we can find $(T_1,T_2,T_1) \in \mathrm{Spin_7}$ such that  $T_1(r_5+r_3.i) = r_7 \in \C$ (and $T_2(r_6) = r_6$, automatically). Hence, we have:
\[A_7 = (T_1,T_2,T_1).A_6 = \begin{pmatrix} {\lambda_1}' & r_6& \overline{b_7} \\ r_6 & \lambda_2 & r_7  \\ b_7 & r_7 & {\l3}'' \end{pmatrix}.\]
We now consider:

\[H_1 = \begin{pmatrix} 1 & 0 & 0 \\ 0 & 1 & 0 \\ - \frac{r_6}{r_7} & 0 & 1 \end{pmatrix}, H_2 = \begin{pmatrix} 1 & 0 & 0 \\ 0 & 1 & \frac{-r_7}{\lambda_2} \\ 0 & 0 & 1 \end{pmatrix}, H_3 =  \begin{pmatrix} -1 & 0 & 0 \\ 0 & 0 & 1 \\ 0 & 1 & 0 \end{pmatrix}.\] We have:

\[ A_8 = {}^{t} H_3{}^{t} H_2 {}^{t} H_1 A_7H_1H_2H_3 =  \begin{pmatrix} \mu_1 & \overline{b_8} &0 \\ b_8 & \mu_2 & 0  \\ 0 & 0 & \mu_3 \end{pmatrix},\]
with $b_8 \in \O$ and $\mu_1, \mu_2, \mu_3 \in \C$.  By lemma \ref{adams}, we can find $(T_1,T_2,T_1) \in \mathrm{Spin_7}$ such that $T_2(\overline{b_8}) = r_8 + r_9.i$, with $r_8,r_9 \in \C$. Hence, we have:

\[A_9 = (T_1,T_2,T_1).A_8 = \begin{pmatrix} \mu_1 & r_8 + r_9.i &0 \\ r_8 - r_9.i & \mu_2 & 0  \\ 0 & 0 & \mu_3 \end{pmatrix}.\]
We let $H = \begin{pmatrix} 0 & 0 & 1 \\ 0 & -1 & 0 \\ 1 & 0 & 0 \end{pmatrix} \in \mathrm{SL}_3$. We have:

\[ A_{10} = {}^{t} H. A_9.H =  \begin{pmatrix} \mu_3 & 0 &0 \\ 0 & \mu_2 & r_8+r_9.i  \\ 0 & r_8-r_9.i & \mu_1 \end{pmatrix}.\]
By lemma \ref{adams}, we can find $(T_1,T_2,T_1) \in \mathrm{Spin_7}$ such that $T_1(r_8+r_9.i) = r_{10}$, with $r_{10} \in \C$. Hence, we have:
\[A_{11} = (T_1,T_2,T_1).A_{10} = \begin{pmatrix} \mu_3 & 0 &0 \\ 0 & \mu_2 & r_{10}  \\ 0 & r_{10} & \mu_1 \end{pmatrix}.\]
\noindent Notice that $A_{11}$ is symmetric with complex coefficients. Since $A \in \J_{3}(\O)$ was chosen generic, we have $\det(A_{11}) = \det_{\O}(A) \neq 0$ and we deduce that there exists $H \in \mathrm{GL}_3$ such that:
\[ {}^{t} H A_{11} H = \begin{pmatrix} 1 & 0 & 0 \\ 0 & 1 & 0 \\ 0 & 0 & 1 \end{pmatrix}.\]
We thus have proved that a generic element $A \in \J_{3}(\O)$ is in the same orbit as $ \begin{pmatrix} 1 & 0 & 0 \\ 0 & 1 & 0 \\ 0 & 0 & 1 \end{pmatrix}$ for the action of the subgroup of $\mathrm{GL}_{24}$ generated by $\mathrm{Spin}_7$ and $\mathrm{GL}_3$. As a consequence, the action of $ \mathbb{C}^* \times \mathrm{Aut}^{0}(\mathfrak{S})$ on $\J_{3}(\O)$ is prehomogeneous.

\end{proof}

\end{subsection}

\end{section}

\newpage

\appendix

\begin{section}{On the dimension the Lie-algebra $\mathfrak{aut}(\SS)$}
In this section, we provide a \textit{Macaulay2} code to get an upper bound on the dimension of $\mathfrak{aut}(\SS)$. We refer to the proof of Theorem \ref{mopiii1} in section $2$, where this computation is used to determine the neutral component of $\mathrm{Aut}(\SS)$.

\bigskip

We first need to write explicitly the equation of $\SS$ in $\textit{Macaulay2}$. This is somehow the main challenge, as the octonion algebra can not be simulated in \textit{Macaulay2}.

\begin{verbatim}
kk = ZZ/313

V = kk[x_1..x_8,y_1..y_27]

M=matrix{{x_1,-x_2,-x_3,-x_4,-x_5,-x_6,-x_7,-x_8},
{x_2,x_1,-x_4,x_3,-x_6,x_5,x_8,-x_7},
{x_3,x_4,x_1,-x_2,-x_7,-x_8,x_5,x_6},
{x_4,-x_3,x_2,x_1,-x_8,x_7,-x_6,x_5},
{x_5,x_6,x_7,x_8,x_1,-x_2,-x_3,-x_4},
{x_6,-x_5,x_8,-x_7,x_2,x_1,x_4,-x_3},
{x_7,-x_8,-x_5,x_6,x_3,-x_4,x_1,x_2},
{x_8,x_7,-x_6,-x_5,x_4,x_3,-x_2,x_1}}

I8 = matrix{{1,0,0,0,0,0,0,0},{0,1,0,0,0,0,0,0},{0,0,1,0,0,0,0,0},{0,0,0,1,0,0,0,0},
{0,0,0,0,1,0,0,0},{0,0,0,0,0,1,0,0},{0,0,0,0,0,0,1,0},{0,0,0,0,0,0,0,1}}

N = matrix{{x_1,-x_2,-x_3,-x_4,-x_5,-x_6,-x_7,-x_8},
{x_2,x_1,x_4,-x_3,x_6,-x_5,-x_8,x_7},
{x_3,-x_4,x_1,x_2,x_7,x_8,-x_5,-x_6},
{x_4,x_3,-x_2,x_1,x_8,-x_7,x_6,-x_5},
{x_5,-x_6,-x_7,-x_8,x_1,x_2,x_3,x_4},
{x_6,x_5,-x_8,x_7,-x_2,x_1,-x_4,x_3},
{x_7,x_8,x_5,-x_6,-x_3,x_4,x_1,-x_2},
{x_8,-x_7,x_6,x_5,-x_4,-x_3,x_2,x_1}}

L = mutableMatrix(M)

for k from 0 to 7 do (for i from 0 to 7 do 
(for j from 1 to 8 do L_(k,i)= sub(L_(k,i), x_j=> y_(j+8))))

S2 = matrix(L)

L = mutableMatrix(N)

for k from 0 to 7 do (for i from 0 to 7 do 
(for j from 1 to 8 do L_(k,i)= sub(L_(k,i), x_j=> y_(j+8))))

R2 = matrix(L)

L = mutableMatrix(M)
for k from 0 to 7 do (for i from 0 to 7 do 
(for j from 1 to 8 do L_(k,i)= sub(L_(k,i), x_j=> y_(j+16))))

S3 = matrix(L)

L = mutableMatrix(N)

for k from 0 to 7 do (for i from 0 to 7 do 
(for j from 1 to 8 do L_(k,i)= sub(L_(k,i), x_j=> y_(j+16))))

R3 = matrix(L)

L = mutableMatrix(M)

for k from 0 to 7 do (for i from 0 to 7 do 
(for j from 1 to 8 do L_(k,i)= sub(L_(k,i), x_j=> y_j)))

S1 = matrix(L)

L = mutableMatrix(N)

for k from 0 to 7 do (for i from 0 to 7 do 
(for j from 1 to 8 do L_(k,i)= sub(L_(k,i), x_j=> y_j)))

R1 = matrix(L)

L = mutableMatrix(M)

for k from 0 to 7 do (for i from 0 to 7 do L_(k,i)= sub(L_(k,i),
 x_1 =>y_1*y_17-y_2*y_18-y_3*y_19-y_4*y_20-y_5*y_21-y_6*y_22-y_7*y_23-y_8*y_24))

for k from 0 to 7 do (for i from 0 to 7 do L_(k,i)= sub(L_(k,i), 
x_2 =>y_1*y_18+ y_2*y_17+y_3*y_20-y_4*y_19+y_5*y_22-y_6*y_21-y_7*y_24+y_8*y_23))

for k from 0 to 7 do (for i from 0 to 7 do L_(k,i)= sub(L_(k,i), 
x_3 =>y_1*y_19-y_2*y_20+y_3*y_17+y_4*y_18+y_5*y_23+y_6*y_24-y_7*y_21-y_8*y_22))

for k from 0 to 7 do (for i from 0 to 7 do L_(k,i)= sub(L_(k,i), 
x_4 =>y_1*y_20+y_2*y_19-y_3*y_18+y_4*y_17+y_5*y_24-y_6*y_23+y_7*y_22-y_8*y_21))

for k from 0 to 7 do (for i from 0 to 7 do L_(k,i)= sub(L_(k,i), 
x_5 =>y_1*y_21-y_2*y_22-y_3*y_23-y_4*y_24+y_5*y_17+y_6*y_18+y_7*y_19+y_8*y_20))

for k from 0 to 7 do (for i from 0 to 7 do L_(k,i)= sub(L_(k,i), 
x_6 =>y_1*y_22+y_2*y_21-y_3*y_24+y_4*y_23-y_5*y_18+y_6*y_17-y_7*y_20+y_8*y_19))

for k from 0 to 7 do (for i from 0 to 7 do L_(k,i)= sub(L_(k,i), 
x_7=> y_1*y_23+y_2*y_24+y_3*y_21-y_4*y_22-y_5*y_19+y_6*y_20+y_7*y_17-y_8*y_18))

for k from 0 to 7 do (for i from 0 to 7 do L_(k,i)= sub(L_(k,i), 
x_8=>y_1*y_24-y_2*y_23+y_3*y_22+y_4*y_21-y_5*y_20-y_6*y_19+y_7*y_18+y_8*y_17))

S13 = matrix(L) 

for k from 0 to 7 do (for i from 0 to 7 do (for j from 2 to 8 do 
L_(k,i)= sub(L_(k,i), y_j=>-y_j)))

for k from 0 to 7 do (for i from 0 to 7 do (for j from 17 to 24 do 
L_(k,i)= sub(L_(k,i), y_j=> y_(j-8))))

Sb12 = matrix(L) 

L = mutableMatrix(Sb12) 

for k from 0 to 7 do (for i from 0 to 7 do (for j from 1 to 8 do 
L_(k,i)= sub(L_(k,i), y_j=> y_(16+j))))

for k from 0 to 7 do (for i from 0 to 7 do (for j from 9 to 16 do 
L_(k,i)= sub(L_(k,i), y_j=> y_(j-8))))

for k from 0 to 7 do (for i from 0 to 7 do (for j from 2 to 8 do 
L_(k,i)= sub(L_(k,i), y_j=> -y_j)))

Sb3b1 = matrix(L) 

for k from 0 to 7 do (for i from 0 to 7 do (for j from 17 to 24 do 
L_(k,i) = sub(L_(k,i), y_j=> y_(j-8))))

for k from 0 to 7 do (for i from 0 to 7 do (for j from 2 to 8 do 
L_(k,i)= sub(L_(k,i), y_j=> -y_j)))

Sb21 = matrix(L)

L = mutableMatrix(Sb12) 

for k from 0 to 7 do (for i from 0 to 7 do (for j from 2 to 8 do 
L_(k,i) = sub(L_(k,i), y_j=> -y_j)))

S12 = matrix(L)

for k from 0 to 7 do (for i from 0 to 7 do (for j from 10 to 16 do 
L_(k,i) = sub(L_(k,i), y_j=> -y_j)))

S1b2 = matrix(L)

L = mutableMatrix(Sb21) 

for k from 0 to 7 do (for i from 0 to 7 do (for j from 10 to 16 do 
L_(k,i) = sub(L_(k,i), y_j=> -y_j)))

S21 = matrix(L)





AA = (S13)*S2 + transpose(S2)*(Sb3b1)

P0 = AA_(0,0)

P1 = -y_27*((S1*transpose(S1))_(0,0))

P2 = -y_26*((S2*transpose(S2))_(0,0))

P3 = -y_25*((S3*transpose(S3))_(0,0))

P = y_25*y_26*y_27 + P0 + P1 + P2 + P3

PHI123 =  (1/2)*((S12-S21)*transpose(S3))_(0,0)




L = mutableMatrix(S1)

for k from 0 to 7 do (for i from 0 to 7 do 
L_(k,i) = sub(L_(k,i), y_1=> 0))

M1 = matrix(L)


L = mutableMatrix(S2)

for k from 0 to 7 do (for i from 0 to 7 do 
L_(k,i) = sub(L_(k,i), y_9=> 0))

M2 = matrix(L)


L = mutableMatrix(S3)

for k from 0 to 7 do (for i from 0 to 7 do 
L_(k,i) = sub(L_(k,i), y_17=> 0))

M3 = matrix(L)


G = matrix{{(M1*transpose(M1))_(0,0),(M1*M2)_(0,0),(M1*transpose(M3))_(0,0)},
{(transpose(M2)*transpose(M1))_(0,0),(M2*transpose(M2))_(0,0),
(transpose(M2)*transpose(M3))_(0,0)},
{(M3*transpose(M1))_(0,0),(M3*M2)_(0,0),(M3*transpose(M3))_(0,0)}}

ASS123square = 4*(det(G) - (PHI123)*(PHI123))

SODM = P^2 - 4*(PHI123)*P - ASS123square
\end{verbatim}

\noindent Here $\texttt{SODM}$ is the equation of the sextic $\SS$, $\texttt{P}$ is $\mathrm{Det}_{\O}$, $\texttt{PHI123}$ is $\phi(c,b,a)$ and $\texttt{ASS123square}$ is $ \bigg|[c,b,a] \bigg|^2$. We have used the formula (see lemma 6.56 and 6.61 in \cite{cali}):

\[ \bigg|[c,b,a] \bigg|^2 = 4 \bigg(\mathrm{Gram}(\mathrm{Im}(c),\mathrm{Im}(b), \mathrm{Im}(a)) - \phi(c,b,a)^2 \bigg),\]
which allows to get the expression of $\bigg|[c,b,a] \bigg|^2$ without computing triple products explicitly in $\O$. We recall that differential of the map $\varphi$ introduced in the proof of Theorem \ref{mopiii1} at a point $\left[(m)_{i,j} \right] \in M_{6 \times 27}$ is given by :
 
\begin{equation*}
\begin{split}
d \varphi_{\left[m_{i,j}\right]} \,\,\, : \,\,\, & M_{6 \times 27} / \mathfrak{aut}(\SS) \barrow S^6 \mathbb{C}^{6} \\
         & \left[q_{i,j} \right] \barrow \sum_{i=1}^{27} \left( \sum_{j=1}^{6} q_{i,j}.z_j \right) \times \dfrac{\SS}{\partial x_i}(m_{1,1}.z_1 + \cdots + m_{1,6}.z_6, \cdots, m_{27,1}.z_1 + \cdots + m_{27,6}.z_6),\\
\end{split}
\end{equation*}

We produce a matrix representing $d \varphi$ at a random point $\left[(m)_{i,j} \right] \in M_{6 \times 27}$.

\begin{verbatim}
D = mutableMatrix(V,1,27)

for i from 1 to 27  do D_(0,i-1) = diff(y_i,SODM)

M = random(ZZ^27,ZZ^6,Height=>50)

M2 = M**V

m = M2 * matrix{{x_1},{x_2},{x_3},{x_4},{x_5},{x_6}}

for i from 0 to 26 do (for j from 1 to 27 do D_(0,i) = sub(D_(0,i), y_j => m_(j-1,0)))

W = kk[x_1..x_6]

Dphi = mutableMatrix(W,1,27)

for j from 0 to 26 do Dphi_(0,j) = sub(D_(0,j),W)

\end{verbatim}

\noindent The matrix $\texttt{Dphi}$ represents $d \varphi$ at the random $\left[(m)_{i,j} \right] \in M_{6 \times 27}$ chosen. We are left to compute the rank of the image of $\texttt{Dphi}$ in $S^6 \mathbb{C}^6$.

\begin{verbatim}

I = (0)

for j from 0 to 26 do (for i from 1 to 6 do I = I + x_i*ideal(DphiW_(0,j)))

II = module I

III = II**kk 

rank III 

\end{verbatim}

\noindent The answer (=133) is obtained on a portable workstation in less than 15 minutes. The semi-continuity Theorem insures that the rank of $d \varphi$ at a generic point $\left[(m)_{i,j} \right] \in M_{6 \times 27}$ and in characteristic $0$ is bigger or equal to $133$.
\end{section}

\newpage
\bibliographystyle{alpha}
\bibliography{Cartan-rank4}
\end{document}